\documentclass[
onefignum,
onetabnum]{siamart190516}

\usepackage{amssymb, amsmath, amsfonts, amscd, graphicx}
\usepackage[margin=1in]{geometry}
\usepackage{enumerate}

\usepackage{tikz}
\usetikzlibrary{positioning}

\usepackage{algorithmic}
\usepackage{amsopn}
\usepackage{epstopdf}

\usepackage{hyperref}
\hypersetup{
    colorlinks=true,
    linkcolor=cyan!50!black!100!, 
    citecolor=cyan,
    urlcolor=cyan, 
}

\newsiamthm{thm}{Theorem}
\newsiamthm{prop}{Proposition}
\newsiamremark{def2}{Definition}

\usepackage{enumitem}
\setlist{nosep}

\newcommand{\ov}{\overline}
\newcommand{\eps}{\epsilon}
\newcommand{\tx}{\textstyle}
\newcommand{\dis}{\displaystyle}

\newcommand{\pr}{\prime}

\newcommand{\al}{\alpha}

\newcommand{\rw}{\rightarrow}

\newcommand{\R}{\mathbb R}

\newcommand{\Tcs}{T_{cS} }
\newcommand{\Tcn}{T_{cN} }
\newcommand{\Tcnp}{T^+_{cN} }
\newcommand{\Tcnm}{T^-_{cN} }

\newcommand{\Tcsm}{T^-_{cs} }
\newcommand{\dotp}{\cdot}

\headers{Synchronous Glacial Cycles}{A. Nadeau, J. Walsh, and E. Widiasih}

\title{Synchronous Glacial Cycles in a Nonsmooth Conceptual Climate Model with Asymmetric Hemispheres}
\author{Alice Nadeau\thanks{a.nadeau@cornell.edu; Department of Mathematics, Cornell University} \and James Walsh\thanks{Department of Mathematics, Oberlin College} \and  Esther Widiasih\thanks{Math, Natural, and Health Sciences (MNHS) Division, University of Hawaii West-Oahu} }

\begin{document}

\maketitle

\begin{abstract}
We present a new conceptual model of the Earth's glacial-interglacial cycles, one leading to  governing equations for which the vector field has a hyperplane of discontinuities. This work extends the classic Budyko- and Sellers-type conceptual energy balance models of temperature-albedo feedback by removing the standard assumption of planetary symmetry about the equator. The dynamics of separate Northern and Southern Hemisphere ice caps are coupled to an 
equation representing the annual global mean surface temperature. The system has a discontinuous switching mechanism  based on mass balance principles for the Northern Hemisphere ice sheet. We show the associated Filippov system admits a unique nonsmooth and attracting limit cycle that represents the cycling between glacial and interglacial states. Due to the vastly different time scales involved, the model presents a nonsmooth geometric perturbation problem, for which we use
 ad hoc mathematical techniques to produce the periodic orbit. We find climatic changes in the Northern Hemisphere drive synchronous changes in the Southern Hemisphere, as is observed for the Earth on orbital time scales. 
\end{abstract}



\section{Introduction}

Systems of nonsmooth differential equations have been used to model a wide range of physical, biological, and mechanical phenomena (see for example the references in \cite{dibernardo}).  In some cases, the nonsmoothness in the models comes from assuming a limiting behavior of an abrupt transition (e.g. in \cite{welander,julie}) while in others the modeled behavior is truly discontinuous or nonsmooth, for example due to friction or impacts in mechanical systems.
Because of the possibility of these types of phenomena in a host of different aspects of the climate system, the mathematical approach to the design and analysis of conceptual climate models increasingly uses tools from the developing field of nonsmooth dynamical systems. Frequently, climate models of this type contain a  switching mechanism  that causes the system to flip to a different climate state. Here we contribute to this body of literature by developing and analyzing a nonsmooth ODE model of Northern and Southern Hemisphere glacial cycles where a hyperplane in state space delineates a switch between the climate state of advancing Northern Hemisphere glaciers and the climate state of retreating Northern Hemisphere glaciers.  For appropriate choices of the system parameters, we show the existence of an attracting periodic orbit corresponding to synchronous Northern and Southern Hemisphere glacial cycles driven by mass balance of the Northern Hemisphere glaciers.

In this study, we consider piecewise smooth systems of the form
\begin{align} 
    \dot{\mathbf{v}}\in\mathbf{X}(\mathbf{v})=\begin{cases}
    {\bf X}_-({\bf v}), &  {\bf v}\in S_-\\ 
    \{ (1-p){\bf X}_-({\bf v})+p{\bf X}_+({\bf v}) : p\in [0,1]\}, & {\bf v}\in\Sigma \\ 
    {\bf X}_+({\bf v}), &  {\bf v}\in S_+ 
    \end{cases}
    \label{eq-intro}
\end{align}
where $\mathbf{v}\in\mathbb R^n$, $\Sigma$ denotes the \emph{switching manifold},  and $X_{\pm}$ is smooth on $S_{\pm}\subset\mathbb R^n$.  Systems of this form are  \emph{differential inclusions}.  While in $S_-$, solutions are unique with flow $\phi_-({\bf v},t)$ corresponding to system  $\dot{\bf v}={\bf X}_-({\bf v})$. Similarly, solutions in 
$S_+$ are unique with flow $\phi_+({\bf v},t)$ given by system $\dot{\bf v}={\bf X}_+({\bf v})$. For ${\bf v}\in\Sigma, \ \dot{\bf v}$ must lie in the closed convex hull of the two vectors ${\bf X}_-({\bf v})$ and ${\bf X}_+({\bf v})$.  A solution to \eqref{eq-intro} {\em in the sense of Filippov} is an absolutely continuous function ${\bf v}(t)$ satisfying $\dot{\bf v}\in {\bf X}({\bf v})$ for almost all $t$ \cite{fil}.

In piecewise systems of the form \eqref{eq-intro}, periodic orbits may be present even when the vector fields $\mathbf{X}_\pm$ don't themselves have periodic orbits (as is the case here).  Heuristically, this happens when the flow in $S_+$ dictates that the solution should cross into $S_-$ and that the flow in $S_-$ dictates that the solution should cross into $S_+$.  More concretely, consider the simplest case: planar systems formed by two continuous differential systems separated by a straight line (as is the situation in, for example, \cite{julie,freire,llibre2017}).  In the simplest cases, periodic orbits can be found by examining the system for different arrangements of equilibria.  In nonsmooth systems there are three types of equilibria to consider, namely\\
\begin{def2}[\cite{dibernardo}]  Let ${\bf v}$ be a solution in the sense of Filippov to \eqref{eq-intro}. 
\begin{enumerate}
\item[(i)] ${\bf v}$ is a {\em regular equilibrium point} of \eqref{eq-intro} if either ${\bf X}_+({\bf v})={\bf 0}$ and ${\bf v}\in S_+$, or if ${\bf X}_-({\bf v})={\bf 0}$ and ${\bf v}\in S_-$.
\item[(ii)]  ${\bf v}$ is a {\em virtual equilibrium point} of \eqref{eq-intro} if either ${\bf X}_+({\bf v})={\bf 0}$ and ${\bf v}\in S_-$, or if ${\bf X}_-({\bf v})={\bf 0}$ and ${\bf v}\in S_+$.
\item[(iii)]  ${\bf v}$ is a {\em boundary equilibrium point} of \eqref{eq-intro} if  ${\bf X}_+({\bf v})={\bf X}_-({\bf v})={\bf 0}$ and ${\bf v}\in \Sigma$.\\
\end{enumerate}
\end{def2}
\noindent Periodic orbits can be present when $\mathbf{X}_{\pm}$  both have regular equilibria (e.g. \cite{freire}), both have virtual equilibria (e.g. \cite{julie,julie-dis} and here), or both have no equilibria of any type (e.g. \cite{llibre2017}). More complicated behavior can also produce nonsmooth periodic orbits, such as when a periodic orbit of either $\mathbf{X}_-$ or $\mathbf{X}_+$ intersects the switching manifold (e.g. Section 2.4 of \cite{dibernardo}).

There are many techniques that can be employed to establish the existence of a nonsmooth periodic orbit in a system of the form of \eqref{eq-intro}.  
The technique of \emph{regularization} is a standard technique which converts the nonsmooth system to a smooth one and can thereby be studied using standard dynamical systems techniques (e.g. \cite{sotomayor1996,dieci,awrejcewicz}).  However, in studying applications other techniques are frequently used because, for example, the smoothing function in the regularization method does not have an explicit form, among other issues limiting the technique's usefulness in specific settings (see the discussion in Section 5.2  in \cite{julie-dis}).  Recent work in showing the existence of periodic orbits in conceptual climate models with switching mechanisms have employed coordinate changes to investigate behavior near the discontinuity boundary \cite{julie-dis,julie} (similar to blow-up techniques in celestial mechanics to investigate behavior near a collision), Filippov's existence and uniqueness results for differential inclusions \cite{Barry2017},  concatenation of smooth solutions from the associated subregions of the phase space where the vector field is smooth \cite{Barry2017}, and construction of a return map on the discontinuity boundary \cite{wwhm,budd}.  

Here we employ the technique used in \cite{wwhm} but in one higher dimension. In particular, we construct a return map on the discontinuity boundary and show that the map is contracting for appropriately chosen parameter values.  This allows us to conclude that there is a unique attracting periodic orbit in the system for appropriate choices of parameter values.  We then demonstrate that for physically relevant choices of parameter values the periodic orbit exists.  This study is an extension of \cite{wwhm} because we have relaxed a symmetry assumption of the climate system about the equator, adding an additional dimension to state space. That is, our model couples separate Northern Hemisphere and Southern Hemisphere ice cap dynamics via the influence that the (possibly asymmetric) positioning of the ice caps has on the global mean surface temperature, and vice versa.  This symmetry assumption is standard in the family of models that we consider (e.g. \cite{sellers,budyko,Widiasih2013,McGehee2014}) but has recently been removed to study the climate of Pluto \cite{Nadeau2019}.  A more general mathematical study of the model with the symmetry assumption removed (but without the mass balance switch that we study here) is forthcoming.

The rest of the paper is laid out as follows.  In the following section we motivate the scientific aspect of this study and describe the physical observations that the model behavior reflects.  
The derivation of the model equations, where separate equations modeling dynamic Northern and Southern Hemisphere ice sheets, via consideration of distinct albedo lines $\eta_N$ and $\eta_S$, and a proxy of the global annual mean surface temperature, $w$, is presented in Section \ref{section-governing-eq}.  Leaving consideration of the dynamics of the $(w,\eta_S,\eta_N)$-system on the boundary of state space for future work, we discuss the behavior of this system  off of  the boundary in Section \ref{Section-Symmetry}. In Section \ref{Section-Mass-Balance} the Northern Hemisphere flip-flop glacial cycle model from \cite{wwhm}  is placed in the Northern Hemisphere of our asymmetric model. We prove the existence of a unique attracting periodic orbit representing the glacial-interglacial cycles, with the mathematical techniques used reminiscent of (smooth) geometric singular perturbation theory.  Notably, the flip-flop behavior of the ice sheet in the Northern Hemisphere drives synchronous oscillations of the ice cap in the Southern Hemisphere via the coupling of the two albedo lines with the surface temperature.  This result aligns with the theory that the Southern Hemisphere ice sheet oscillations are in response to climate changes in the Northern Hemisphere on orbital time scales.

\section{Scientific Background}

Understanding the behavior of the glaciers over time and the resulting impact on Earth's climate has been a major endeavour across disparate fields of the physical and biological sciences for over a hundred years. Glacial cycles are characterized by the advance of large ice sheets from the poles to the mid-latitudes and their subsequent retreat and are a defining characteristic of Earth's climate history. Glacier advance, occurring over tens of thousands of years, is not monotonic and glacial records show periods of relative warming as the climate gradually cools to the glacial maximum \cite{broecker, brook, pedro}. Relative to the long time-scale of their advance, glacier retreat is fast, taking only thousands of years instead of tens of thousands (e.g. see \cite{tzip2003}).  This advance and retreat cycle creates a characteristic sawtooth pattern in the glacial record for roughly the past 800,000 years (e.g. see for example \cite{brook}).  

Many questions concerning the Earth's glacial cycles remain unanswered, including those related to changes in the period and amplitude  of the glacial-interglacial cycles that have occurred over geologic time.  More relevant to our model is evidence that on orbital times scales (100 kyr) ice cover oscillations in the Northern  and Southern Hemisphere have been in sync \cite{broecker,brook,lowell,raylisnic,rother}. While the physical mechanisms behind these different behaviors continue to be investigated, some posit  that changes in the Northern Hemisphere climate drive changes in the Southern Hemisphere on orbital time scales \cite{broecker,brook,lowell,raylisnic}. The climate changes in the Northern Hemisphere are in turn thought to be brought about by changes in high northern latitude incoming solar radiation, due to changes in Earth's orbital parameters  over long time scales (the latter known as Milankovitch cycles \cite{milank}).  Studies have demonstrated the prominent role that Earth's obliquity (tilt of the axis of rotation relative to the orbital plane) plays in pacing the glacial cycles \cite{huybers2005}, but the jury is still out on whether precession plays a definitive role (e.g. \cite{huybers2011} and references therein). Crucially,  precession acts with the opposite effect in the hemispheres (e.g. \cite{brook,huybers2011}).

Further, it is not known to what extent glacial-interglacial cycles are precipitated by orbital forcings in conjunction with internal climate feedbacks \cite{brook}.  Such feedbacks include greenhouse gas forcing, albedo (surface reflectivity) feedbacks, dust forcing, deep ocean temperature, isostatic rebound, or mass balance of Northern Hemisphere glaciers \cite{abe-ouchi,brook}.  Here we consider two of these feedback mechanisms, albedo and mass balance, on Earth's surface temperature.  Recent work suggests that these two mechanisms are not necessarily decoupled and  several studies have noted and investigated how changes in a glacier's albedo may influence local temperature or precipitation feedbacks and thus affect a glacier's growth (e.g. \cite{abe-ouchi,raylisnic,tzip2003}).  For instance,   Tziperman and Gildor  note that extensive, high albedo sea ice cools the atmospheric temperature and can divert snow storms away from continental ice sheets (\cite{tzip2003} and references therein).

The model that we consider in this study is a \emph{conceptual climate model} (sometimes \emph{simple climate model}, \emph{low complexity climate model}, \emph{analytical climate model} or \emph{reduced climate model}). Conceptual models are used to give a broad view of the ways in which major climate components interact, contrasting with higher complexity models (such as general circulation models or earth systems models with two or three spatial dimensions) which simulate atmospheric, oceanic, chemical, and biospheric dynamics on a grid of the Earth.  While in the past highly complex climate models have not been applied to study the long-term behavior of the past climate system (due to limited computing power and the length of time series needed to simulate, for example), recently intermediate to high complexity models have been been adapted to successfully study glacial dynamics (e.g. \cite{abe-ouchi,kawamura, choudhury}).  Conceptual climate models still have an important role to play  in advancing scientific understanding of glacial cycles and glacier dynamics (e.g. \cite{budyko,sellers,saltzman,huybers2005,huybers2011,engler2018dynamical}) and are also a more computationally efficient way to test theories about interactions between different climate elements before implementing the idea in a more complex model (e.g. \cite{knutti}). In the case of glacial cycle models the climate elements considered might include  surface temperature, energy into and out of the climate system, the latitudinal transport of energy, the carbon cycle, and the ways in which processes such as surface  albedo affect these interactions.

Conceptual modeling of the glacial cycles using energy balance equations was popularized by the work of M. Budyko \cite{budyko} and W. Sellers \cite{sellers} in 1969, with the introduction of equations used to model surface temperature on a planet with an assumed symmetry about the equator.  The temperature model that we use here is a descendent of Budyko's original equation \cite{budyko}.  Following the through line of the Budyko family of models leading to the model we study here, E. Widiasih coupled Budyko's temperature equation with a dynamic ice sheet in \cite{Widiasih2013}, proving the existence of a small stable ice cap for the  resulting infinite-dimensional system.   An approximation of Widiasih's temperature-albedo line system  was introduced in \cite{McGehee2014}, a simplification using smooth invariant manifold theory that resulted  in a planar system of ODEs  exhibiting the same qualitative behavior as Widiasih's system.  
The approximating  temperature-albedo line system in \cite{McGehee2014} then served as the basis for the nonsmooth ``flip-flop" glacial cycle presented in \cite{wwhm}, in which a nonsmooth attracting periodic orbit was shown to exist. This periodic orbit represented the Earth's climate system cycling between glacial and interglacial states, with the switching mechanism provided by a conceptual ice sheet mass balance principle. In this study, we extend the \cite{wwhm} model by removing a symmetry assumption about the climate system.

This removal of the symmetry assumption is justified when considering the inherent asymmetry between the northern and southern polar regions. The most notable difference is the fact that over the past 800,000 years, Antarctica has been completely glaciated, and ``glacial" advance and retreat refers to major changes in Southern Ocean sea ice extent \cite{raylisnic,gersonde,fraser} and glaciers in mountainous areas of southern South America, Africa, and Oceania \cite{rother,darvill} rather than the large glaciers terminating on land in the Northern Hemisphere \cite{abe-ouchi}. Thus, while glaciers in the Northern Hemisphere terminated on land, Southern Hemisphere glaciers terminated in the Southern Ocean with large ice shelves and sea ice extent reaching perhaps as far as 45$^\circ$S at times in some places \cite{gersonde,fraser}. For this reason, we do not place a mass balance equation in the Southern Hemisphere and allow the Southern Hemisphere albedo line to indicate Southern Hemisphere glacial dynamics in our model. Other potential differences in the glacial records are smaller amplitude oscillations for ice volume in the Southern Hemisphere, relative to the Northern Hemisphere \cite{raylisnic}; however, some records indicate oscillations of similar amplitude \cite{blunier}.

The main question that our model addresses is: do Northern Hemisphere glacial cycles affect the Southern Hemisphere and, if so, can they drive synchronous cycles in both hemispheres?  Here we explicitly consider the role of global temperature/albedo feeback and Northern Hemisphere albedo/mass balance feedback.  We leave the impact of Earth's changing orbital parameters to a later study.

\section{Governing Equations with Two Albedo Lines} \label{section-governing-eq}

\subsection{Temperature Equation}

The energy balance equations introduced by M. Budyko \cite{budyko} and W. Sellers \cite{sellers} in 1969  describe the evolution of the Earth's latitudinally averaged annual mean  surface temperature $T(y,t)$, where $t$ denotes time in years and $y$ denotes the sine of the latitude.  In the model we use here, the temperature evolves based on   M. Budyko's   energy balance equation \cite{budyko}
\begin{equation}
\begin{aligned} 
 R \frac{\partial T}{\partial t} = Q s(y) (1-\alpha (y,\mathbf\eta) ) - \left(A + B T(y,t) \right) - C \left(T(y,t) - \overline{T}(t) \right),
 \end{aligned}
  \label{EQ-Budyko}
  \end{equation}
where the change in temperature is determined by the absorbed solar radiation, $Q s(y,\beta) (1-\alpha (y,\mathbf\eta))$; the emitted longwave radiation, $A + B T(y,t)$; and  energy transport across latitudes, $C \left(T(y,t) - \overline{T}(t) \right)$ where $\overline{T}(t)$ is the global average temperature.  We note W. Sellers independently introduced a similar energy balance model in the same year Budyko's appeared, albeit one with a different meridional energy transport mechanism \cite{sellers}.

The physical meaning of the different terms and parameters of \eqref{EQ-Budyko} have been explained extensively in the literature (see for example \cite{Tung2007,Kaper2013,Widiasih2013,Nadeau-dis}), so we omit a detailed explanation here.  Instead we provide Table \ref{TAB-parameters-Earth} with brief physical descriptions of the parameters and note the major changes due to our removal of the symmetry assumption used in previous studies in the remainder of this section.

Because we consider the possibility of asymmetry between the hemispheres, we let sine of the latitude $y$ range from the south pole $y=-1$ to the north pole $y=1$ rather than from the equator to the north pole ($y\in[0,1]$). The surface albedo is given by $\alpha(y,\eta)$, which depends on $y$ and the location of surface ice, the lower-latitude boundary of which is typically denoted $\eta$.  In this work, however, we take $\eta=(\eta_S,\eta_N)$, which gives the location of a southern ($\eta_S$) and northern ($\eta_N$) latitude where the albedo changes.  We restrict these variables to the interval $[-1,1]$ with the condition $-1\leq\eta_S\leq\eta_N\leq1$ (i.e., we do not let the ice lines cross each other).  We consider a piecewise constant albedo function given by
\begin{equation}\label{alb}
\al(y,\eta_S,\eta_N)=
 \begin{cases}
\al_2, &\text{if  \ } -1 <y<\eta_S \\
\al_1, &\text{if  \ } \eta_S <y<\eta_N  \\
\al_2, &\text{if  \ }  \eta_N<y<1,\\
\end{cases} \quad \quad
\end{equation}
with appropriate averages at the ice lines. 
We take $\alpha_1<\alpha_2$ so that the regions poleward of the ice lines are more reflective.  

The energy transport term is a simple linear relaxation to the mean annual global temperature given by integrating the temperature over all latitudes (the interval $[-1,1]$) $\overline T(t) = \frac{1}{2}\int_{-1}^1T(y,t)dy$.  Finally, for clarity, note that the distribution of the annual insolation across $y$, which also depends on the tilt of the Earth's spin axis (or {\em obliquity})   $\beta$, can be approximated to any degree of accuracy by
\begin{equation}\label{s(y)}
s(y)=\sum^M_{m=0} a_{2m}p_{2m}(\cos\beta)p_{2m}(y),
\end{equation}
 where $p_{2m}$ is the $2m$th Legendre polynomial and the $a_{2m}$ can be explicitly determined following \cite{Nadeau2017}.  In this study we fix the obliquity at the Earth's current value $\beta=23.5^\circ$, for which  $s_{2m}=a_{2m}p_{2m}(\cos(\pi 23.5/180)),$ and we write $s(y)$ in lieu of $s(y,\beta)$. 

In a computation similar to that presented for Budyko's equation in \cite{Tung2007}, one finds that at equilibrium the temperature distribution is 
\begin{equation}\label{Tstar}
T^*(y)=\frac{1}{B+C}\left( Qs(y)(1-\al(y))-A+C\overline{T^*}\right),
\end{equation}
with the global mean temperature  given by
\begin{equation}\label{Tstarbar}
\overline{T^*}=\frac{1}{B}\left(Q(1-\al_2)-A+\frac{1}{2}Q(\al_2-\al_1) \int^{\eta_N}_{\eta_S}s(y) dy\right).
\end{equation}
 We note that, due to the use of expansion \eqref{s(y)}, the equilibrium function $T^*(y)$ is a (``piecewise even") polynomial of degree $2M$ in $y$ and degree $2M+1$ in each of $\eta_N$ and $\eta_S$.
 
 \subsection{Albedo Line Equations}

Here we consider two dynamic ice line equations in the fashion of Widiasih's single ice line equation \cite{Widiasih2013}. In particular, the movement of an ice line is determined by the temperature at the ice line relative to a critical temperature $T_c$, the highest temperature at which ice is present year round.  We have
 \begin{equation}
 \begin{aligned}
 \frac{d\eta_S}{dt}&=\rho(T_c-T(\eta_S,t)),\\
  \frac{d\eta_N}{dt}&=\rho(T(\eta_N,t)-T_c).
 \end{aligned}
 \label{EQ-ice-line-cap}
 \end{equation}
These equations dictate that if the temperature at the albedo line is greater than the critical temperature, the albedo line moves toward its own pole.  If the temperature is less than the critical temperature, the albedo line moves toward the opposite pole.   
The positive parameter $\rho$ controls how fast the ice line changes relative to changes in temperature. 
In their discussion of glacial cycles on Earth, McGehee and Widiasih give an in-depth discussion on the behavior of solutions of a similar, hemispherically symmetric energy balance model relative to the choice of $\rho$ \cite{McGehee2014}.

\begin{table}
\caption{Parameter values used in this study (unless otherwise noted).}
\begin{center}
    \begin{tabular}{ || c | p{6cm} | r | r  ||}
    \hline
    Parameter & Brief Description  & Value & Units  \\ \hline     \hline
    $R$     & Surface layer heat capacity	& 1  & Wm$^{-2}$K$^{-1}$ 	\\ \hline
    $Q$       & Annual average insolation	& 343	& Wm$^{-2}$ \\ \hline
    $\beta$    & Obliquity& 23.5  & degrees \\ \hline
    $\alpha_1$  & Albedo between the albedo line latitudes $\eta_S(t)$ and $\eta_N(t)$	& 0.32   & dimensionless 	\\ \hline
    $\alpha_2$  & Albedo poleward of the albedo line latitudes $\eta_S(t)$ and $\eta_N(t)$	& 0.62  & dimensionless	 \\ \hline
    $A$  & Greenhouse Gas parameter	& 202  & Wm$^{-2}$	 \\ \hline
    $B$  & Outgoing radiation& 1.9  & Wm$^{-2}$K$^{-1}$	 \\ \hline
    $C$  & Efficiency of heat transport& 3.04  &  Wm$^{-2}$K$^{-1}$	 \\ \hline
     $T_c; \ T_{cS/N}, \ T_{cN}^{\pm}$  & Critical temperature determining advance/retreat of albedo lines & -10; varies  & $^\circ$C	 \\ \hline
    $\rho$  & Albedo line response to temperature change	& 0.3  & K$^{-1}$yr$^{-1}$	 \\ \hline
    $2M$  & Degree of the polynomial approximation of the insolation function	& 2  & dimensionless	 \\ \hline
    $a$  & Accumulation rate	& 1.05  & dimensionless \\ \hline
    $b$  & Critical ablation rate	& 1.75  &  dimensionless\\ \hline
    $b_-$  & Glacial ablation rate	& 1.5  &  dimensionless\\ \hline
    $b_+$  & Interglacial ablation rate	&  5 & dimensionless \\ \hline
    $\epsilon$  & Mass balance response to albedo change	& 0.03  & yr$^{-1}$	 \\ \hline
    \end{tabular}
\end{center}
\label{TAB-parameters-Earth}
\end{table}

\subsection{Finite-dimensional approximation of the temperature equation}

Recall the equilibrium temperature distribution \eqref{Tstar} is a piecewise even function of $y$. In addition, we are assuming the expansion of $s(y)$ in even Legendre polynomials \eqref{s(y)}. We are thus motivated to express the temperature function piecewise  as follows:
\begin{align}\label{Texp}
T(y,t)=
\begin{cases}
U(t,y)=\sum^M_{m=0} u_{2m}(t)p_{2m}(y), & -1\leq y <\eta_S\\
V(t,y)=\sum^M_{m=0} v_{2m}(t)p_{2m}(y), & \eta_S< y <\eta_N\\
W(t,y)=\sum^M_{m=0} w_{2m}(t)p_{2m}(y), & \eta_N<y <1.\\
\end{cases}
\end{align}
The use of expression \eqref{Texp} is similar in spirit to that used in \cite{Walsh2015} to model extensive glacial episodes in the Neoproterozoic Era, work in turn motivated by  the approach to Budyko's equation taken in \cite{McGehee2014}. 
The temperature at each ice line is taken to be the appropriate average, namely,
\begin{align}\label{Teta1}
T(\eta_S)&=\tx{\frac{1}{2}}\sum^M_{m=0} (u_{2m}+v_{2m})p_{2m}(\eta_S),\\\notag
T(\eta_N)&=\tx{\frac{1}{2}}\sum^M_{m=0} (v_{2m}+w_{2m})p_{2m}(\eta_N).\notag
\end{align}
Separately substituting each expression in \eqref{Texp} along with expansion \eqref{s(y)} into equation \eqref{EQ-Budyko}, and equating the  respective coefficients of $p_{2m}$, one arrives at the system of $3(M+1)$ ODEs
\begin{align}\label{full}
R\dot{u}_0&= Qs_0(1-\al_2)-A-(B+C)u_0+C\ov{T}  \\  \notag
R\dot{u}_{2m}&  = Qs_{2m}(1-\al_2) -(B+C)u_{2m}, \qquad\qquad  m\geq 1  \\  \notag
R\dot{v}_0&= Qs_0(1-\al_1)-A-(B+C)v_0+C\ov{T}  \\  \notag
R\dot{v}_{2m}&  = Qs_{2m}(1-\al_1) -(B+C)v_{2m}, \qquad\qquad   m\geq 1  \\  \notag
R\dot{w}_0&= Qs_0(1-\al_2)-A-(B+C)w_0+C\ov{T} \\ \notag
R\dot{w}_{2m}&= Qs_{2m}(1-\al_2) -(B+C)w_{2m}, \qquad\qquad   m\geq 1.  \notag
\end{align}
In addition
\begin{align}\label{Tbar1}
2\ov{T}&=\int^{\eta_S}_{-1}U(t,y) dy+\int^{\eta_N}_{\eta_S}V(t,y) dy+\int^{1}_{\eta_N}W(t,y) dy\\\notag
&=\int^{1}_{-1}U(t,y) dy-\int^{\eta_N}_{\eta_S}U(t,y) dy-\int^{1}_{\eta_N}U(t,y) dy+\int^{\eta_N}_{\eta_S}V(t,y) dy+\int^{1}_{\eta_N}W(t,y) dy\\\notag
&=2u_0-\sum^M_{m=0}(u_{2m}-v_{2m})(P_{2m}(\eta_N)-P_{2m}(\eta_S))-\sum^M_{m=0}(u_{2m}-w_{2m})(1-P_{2m}(\eta_N)),
\end{align}
where we set $P_{2m}(y)=\int p_{2m}(y)dy, m\geq 0$ for ease of notation.

Note the decoupling in \eqref{full}; each of the equations tends to equilibrium except for the three equations corresponding to $m=0$. We thus assume that
\begin{equation}\label{uvw*}
u_{2m}=u^*_{2m}=Ls_{2m}(1-\al_2), \ v_{2m}=v^*_{2m}=Ls_{2m}(1-\al_1), \ w_{2m}=w^*_{2m}=u^*_{2m}, \ m\geq 1,
\end{equation}
where we have let  $L=Q/(B+C)$.

With assumption \eqref{uvw*}, equations \eqref{Teta1} become
\begin{align}\label{Teta2}
T(\eta_S)&=\tx{\frac{1}{2}}(u_0+v_0)+\tx{\frac{1}{2}}\sum^M_{m=1} (u^*_{2m}+v^*_{2m})p_{2m}(\eta_S),
\\\notag
T(\eta_N)&=\tx{\frac{1}{2}}(v_0+w_0)+\tx{\frac{1}{2}}\sum^M_{m=1} (v^*_{2m}+u^*_{2m})p_{2m}(\eta_N).\notag
\end{align}
In addition, and after much simplification, \eqref{Tbar1} can be placed in the form
\begin{equation}\label{Tbar2}
2\ov{T}=\eta_S(u_0-v_0)+\eta_N(v_0-w_0)+u_0+w_0+L(\al_2-\al_1)\sum^M_{m=1} s_{2m}(P_{2m}(\eta_N)-P_{2m}(\eta_S)).
\end{equation}

For an additional simplification,  note that if  $x=u_0-w_0$ then
\begin{equation}\notag
R\dot{x}=R\dot{u}_0-R\dot{w}_0=-(B+C)x.
\end{equation}
We have $x(t)=u_0(t)-w_0(t)\rw 0$ as $t\rw\infty$, and hence we assume $u_0=w_0$. Thus in system \eqref{full}, we need only consider the $u_0$- and $v_0$-equations. Also recalling $w^*_{2m}=u^*_{2m}$ for $m\geq 1,$ $u_0=w_0$ additionally implies that 
$U(t,y)$ and $W(t,y)$ are part of the same degree $2M$ polynomial of $y$, albeit with different  domains, again assuming all the appropriate variables are at equilibrium.

We are thus lead to  consider the pair of equations
\begin{align}\label{u0v0}
R\dot{u}_0&= Qs_0(1-\al_2)-A-(B+C)u_0+C\ov{T}  \\  \notag
R\dot{v}_0&= Qs_0(1-\al_1)-A-(B+C)v_0+C\ov{T}.  \notag
\end{align}
For $m\geq 1, \  u^*_{2m}+v^*_{2m}=2Ls_{2m}(1-\al_0)$, where $\al_0=\frac{1}{2}(\al_1+\al_2).$ Equations \eqref{Teta2} then become
\begin{align}\label{Teta3}
T(\eta_S)&=\tx{\frac{1}{2}}(u_0+v_0)+L(1-\al_0)\sum^M_{m=1}s_{2m}p_{2m}(\eta_S)\\\notag
&=\tx{\frac{1}{2}}(u_0+v_0)+L(1-\al_0)(s(\eta_S)-s_0p_0(\eta_S))\\\notag
&=\tx{\frac{1}{2}}(u_0+v_0)+L(1-\al_0)(s(\eta_S)-1), \mbox{ and similarly}\\ \notag
T(\eta_N)&=\tx{\frac{1}{2}}(u_0+v_0)+L(1-\al_0)(s(\eta_N)-1).
\end{align}
Setting $w_0=u_0$ in  \eqref{Tbar2}   yields
\begin{equation}\label{Tbar3}
2\ov{T}=2u_0-(u_0-v_0)(\eta_N-\eta_S)+L(\al_2-\al_1)\left(\int^{\eta_N}_{\eta_S}s(y) dy-(\eta_N-\eta_S)\right).
\end{equation}

As a final step we introduce the new variables   $w=\frac{1}{2}(u_0+v_0)$ and  $z=u_0-v_0$. System \eqref{u0v0} becomes
\begin{align}\label{wz}
R\dot{w}&=Qs_0(1-\al_0)-(B+C)w-A+C\ov{T}\\\notag
R\dot{z}&=Qs_0(\al_1-\al_2)-(B+C)z.\notag
\end{align}
We see that $z\rw z^*=Ls_0(\al_1-\al_2)$ as $t\rw\infty$, and so we set $z=z^*$ in all that follows. We have reduced the study of system \eqref{full} to that of the equation 
\begin{equation}\label{wdot}
R\dot{w}=Qs_0(1-\al_0)-(B+C)w-A+C\ov{T}.
\end{equation}
In terms of $w$, equations \eqref{Teta3} become 
\begin{equation}\label{Teta4}
T(\eta_S)=w+L(1-\al_0)(s(\eta_S)-1),
 \ \ 
T(\eta_N)=w+L(1-\al_0)(s(\eta_N)-1), 
\end{equation}
while \eqref{Tbar3} simplifies to
\begin{equation}\label{Tbar4}
\ov{T}=w-\tx{\frac{1}{2}}Ls_0(\al_2-\al_1)\left(1-\dis \int^{\eta_N}_{\eta_S}s(y) dy\right).
\end{equation} 
Note equation \eqref{Tbar4} states that $w$ is a translation of the global annual mean surface temperature, where the translation depends upon the integral of the insolation distribution function $s(y)$ between the albedo lines.

Coupling the temperature equation \eqref{wdot} with the ice line evolution equations \eqref{EQ-ice-line-cap} gives   a $(w,\eta_S,\eta_N)$-system that can be placed in the form
\begin{equation}
\begin{aligned}
\frac{dw}{dt}&=-\frac{B}{R}\left(w-F(\eta_S, \eta_N)\right)\\
\frac{d\eta_S}{dt}&=-\rho(w-G(\eta_S))\\
\frac{d\eta_N}{dt}&=\rho(w-G(\eta_N)),
\end{aligned}
\label{system}
\end{equation}
where
\begin{equation}\label{F}
F(\eta_S,\eta_N)=\frac{1}{B}\left(Qs_0(1-\al_0)-A+\frac{1}{2}CLs_0(\al_1-\al_2)(1-\dis \int^{\eta_N}_{\eta_S}s(y) dy)\right),
\end{equation}
and\begin{equation}\label{G}
G(\cdot)=-L(1-\al_0)(s(\cdot)-1)+T_c.
\end{equation}

\section{Behavior of the Two Albedo Line System}
\label{Section-Symmetry}

We begin with a discussion of the case in which the critical temperatures at $\eta_S$ and $\eta_N$ are equal, with each denoted $T_c$. We then discuss how different critical temperatures affect the equilibria of the system. This discussion portends analysis to follow in Section \ref{Section-Mass-Balance}. In the full model with the mass balance equations, the critical temperature at the northern albedo line will change depending on whether the climate state is in a glacial period or an interglacial period. 

We consider our system \eqref{system} on the space
\begin{equation}\notag
{\mathcal B}^\pr=\{ (w,\eta_S,\eta_N) : w\in\R, \eta_S,\eta_N\in[-1,1], \eta_S\leq \eta_N \}.
\end{equation}
The restriction of $\eta_S$ and $\eta_N$ to $[-1,1]$ corresponds to the physical boundary of the latitudes at the south and north poles.  The boundary component given by $\eta_S\leq \eta_N$ ensures that we do not have the (nonphysical) situation of the albedo lines crossing (the case where $\eta_S=\eta_N$ indicates a snowball Earth).  In a subsequent paper  system \eqref{system} will be analyzed on the boundary of ${\mathcal B}^\pr$ via the introduction of an appropriately defined Filippov flow (akin in spirit to 
\cite{Barry2017}). In the present work we restrict attention to the interior of the state space; nonetheless, a detailed description of Fillipov flows will be presented in Section \ref{Section-Mass-Balance}, in which their use is needed to analyze a (discontinuous) extension of \eqref{system} in which separate albedo and snow lines are considered.

\subsection{Equal critical temperatures} \label{Section-Tc-Equal}

Let ${\bf Y}_+={\bf Y}_+(w,\eta_S,\eta_N)$ denote the vector field given in \eqref{system}, with  $\psi_+=\psi_+((w,\eta_S,\eta_N),t)$ its associated flow. (The use of the subscript $+$ foreshadows analysis to come in Section \ref{Section-Mass-Balance}.)
 We set $M=1$ in \eqref{s(y)} in all that follows as the use of higher order approximations   yields qualitatively similar results. 

With parameters as in Table 1, system \eqref{system}   has two equilibria in ${\mathcal B}^\pr$ given by
\begin{align*}
Q^u_+&=(w^u_+,(\eta_S)^u_+,(\eta_N)^u_+)=( -17.118, -0.249, 0.249) \ \mbox{ and } \\ Q^s_+&=(w^s_+,(\eta_S)^s_+,(\eta_N)^s_+)=(5.188, -0.955, 0.955),
\end{align*}
each lying in the plane $\eta_N=-\eta_S$. 
As the Jacobian $J{\bf Y}_+(Q^s_+)$ has  eigenvalues $-15.85, -15.05$ and  $-1.10$,   $Q^s_+$ is a stable node for the flow $\psi_+$.  (One can check equilibrium  $Q^u_+$ is a saddle having 2-dimensional stable manifold.)  Note the equilibrium $Q^s_+$ corresponds  to small, symmetric ice caps,  while    $Q^u_+$ corresponds to (unstable) large, symmetric ice caps.  These results agree with earlier studies where the albedo lines are assumed to be symmetric across the equator (e.g., \cite{McGehee2014,Widiasih2013}).

To help visualize these structures, we plot the $w$-nullcline for ${\bf Y}_+$ (green), together with the  curve of intersection of the $\eta_S$- and $\eta_N$-nullclines for ${\bf Y}_+$ (red) in Figure \ref{FIG-Y-structures}. The intersection of the red curve and the green surface yields the two  equilibria in ${\mathcal B}^\pr$ for \eqref{system}. Also plotted in Figure \ref{FIG-Y-structures}
is the projection of the curve of intersection of the $\eta_S$- and $\eta_N$-nullclines for ${\bf Y}_+$ (red) in the $\eta_S\eta_N$-plane, which can be shown to be  the line $\eta_N=-\eta_S$.

We pause to comment on the role played by the parameter $T_c$, which appears in the $\dot{\eta}_S$- and 
$\dot{\eta}_N$-equations in system \eqref{system}. An increase in $T_c$ serves to translate the $\eta_S$- and $\eta_N$-nullclines up, that is, the red curve in the left plot in Figure  \ref{FIG-Y-structures} moves up while the $w$-nullcline remains unchanged. This causes $Q^s_+$ and $Q^u_+$ to move towards each other ($(\eta_S)_{+}^s$ and $(\eta_N)_{+}^s$ move symmetrically toward the equator, $(\eta_S)_{+}^u$ and $(\eta_N)_{+}^u$ move symmetrically toward their respective poles), corresponding to larger stable ice caps at equilibrium. A sufficiently large increase in $T_c$ leads to a saddle-node bifurcation in which the $\eta_S$- and $\eta_N$-nullclines tangentially intersect the $w$-nullcline before passing above the $w$-nullcline.

Similarly, a decrease in $T_c$ from $-10^\circ$C moves $Q^s_+$ and $Q^u_+$  away from each other. A sufficiently negative $T_c$ first leads to $(\eta_S)_{+}^s=-1$ and $(\eta_N)_{+}^s=1$, corresponding to a ``stable" ice-free Earth. Further decreasing $T_c$ leads to $Q_{+}^s$ leaving $\mathcal{B}'$ and, eventually, $(\eta_S)_{+}^u=(\eta_N)_{+}^u$, corresponding to an ``unstable'' completely glaciated Earth. An even further decrease in $T_c$ causes $Q_{+}^u$ to leave $\mathcal{B}'$ as well. Formalizing these statements requires consideration of the dynamics on the boundary of ${\mathcal B}^\pr$, which will appear in future work.  We note the range of $T_c$-values used in the following section ensures the existence of two equilibria for system \eqref{system} within the interior of ${\mathcal B}^\pr$.

\begin{figure}
\begin{center}
\includegraphics[width=6in,trim = 1.5in 7in 1in  1in, clip]{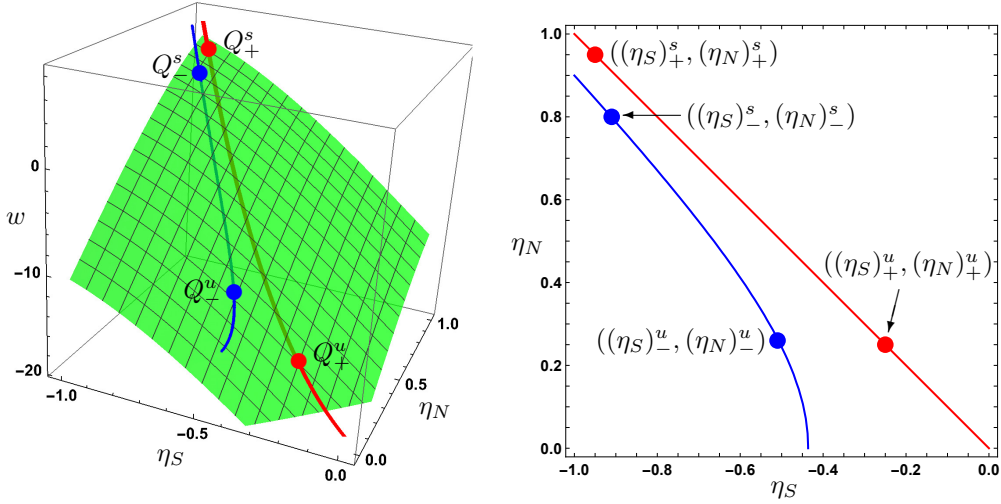}\\
\caption{{ {\em Left}: The $w$-nullcline (green), and the curves of intersection of the   of the $\eta_S$- and $\eta_N$-nullclines for ${\bf Y}_+$ (red) when $\Tcsm=\Tcnm=-10^\circ$C and ${\bf Y}_-$ (blue) when $\Tcsm=-10^\circ$C and $\Tcnm=-5^\circ$C. \ {\em Right}: The projections of  the curves of intersection     of the $\eta_S$- and $\eta_N$-nullclines for ${\bf Y}_+$ (red) and ${\bf Y}_-$ (blue) in the $\eta_S\eta_N$-plane.
}} 
\label{FIG-Y-structures}
 \end{center}
\end{figure}

\subsection{Different critical temperatures}

While the critical temperature value $T_c=-10^\circ$C is often used in the energy balance climate literature for the Earth, other values have been used as well. For example, $T_c$ was set to $0^\circ$C in \cite{pierre} when modeling a generally colder world. A linear drift  in $T_c$ from $-13^\circ$C
to $-3^\circ$C was incorporated in the glacial cycle model presented in \cite{tzip2003} to represent the cooling of the deep ocean during the Pleistocene. 

We thus consider the case in which the critical temperature $T_{cS}$ at $\eta_S$ differs from the critical temperature $T_{cN}$ at $\eta_N$, a possibility easily investigated with our model. Consider the system
\begin{equation}
\begin{aligned}
\frac{dw}{dt}&=-\frac{B}{R}\left(w-F(\eta_S, \eta_N)\right)\\
\frac{d\eta_S}{dt}&=-\rho(w-G_S(\eta_S))\\
\frac{d\eta_N}{dt}&=\rho(w-G_N(\eta_N)),
\end{aligned}
\label{system-NS}
\end{equation}
where
\begin{equation}\label{GS}
G_S(\eta_S)=-L(1-\al_0)(s(\eta_S)-1)+T_{cS}\quad\text{and}\quad G_N(\eta_N)=-L(1-\al_0)(s(\eta_N)-1)+T_{cN}.
\end{equation}
We let ${\bf Y}_-={\bf Y}_-(w,\eta_S,\eta_N)$ denote the vector field given in system \eqref{system-NS}, for which $w,\eta_S,\eta_N\in\mathcal B^\pr$ and the parameters are given in  Table 1, with the sole exception being that we allow  $T_{cN}>-10^\circ$C. (The use of the subscript `$-$' will become clear in Section \ref{Section-Mass-Balance}.)
We let $\psi_-=\psi_-((w,\eta_S,\eta_N),t)$ denote the flow associated with \eqref{system-NS}.

The scenario $T_{cS}=-10^\circ$C and $T_{cN}=-5^\circ$C is depicted in Figure \ref{FIG-Y-structures}. The green $w$-nullcline remains unchanged as the critical temperature does not appear in the $\dot{w}$-equation. Recall the red curve in the left plot in Figure \ref{FIG-Y-structures} is the intersection of the $\eta_S$- and $\eta_N$-nullclines in the symmetric case ($T_{cS}=T_{cN}=-10^\circ$C in \eqref{system-NS}). The blue curve in the left plot in Figure \ref{FIG-Y-structures} is the intersection  of the $\eta_S$- and $\eta_N$-nullclines for \eqref{system-NS} when $T_{cS}=-10^\circ$C and $T_{cN}=-5^\circ$C. Also plotted in Figure  \ref{FIG-Y-structures} are the projections of the red and blue curves in the 
$\eta_S\eta_N$-plane.

Keeping $T_{cS}=-10^\circ$C fixed, we see in Figure \ref{FIG-Y-structures} that  an increase of $T_{cN}$ from $-10^\circ$C to $-5^\circ$C yields an equilibrium  point $Q^s_{-}=(w^s_-,(\eta_S)^s_-,(\eta_N)^s_-)$ for ${\bf Y}_{-}$ near $Q^s_+$    with $(\eta_S)^s_->(\eta_S)^s_+$ and $(\eta_N)^s_-<(\eta_N)^s_+$ (that is, each albedo line has moved equatorward). Given that $Q^s_+$ is a stable node for ${\bf Y}_+$, and using the fact ${\bf Y}_+$ and ${\bf Y}_{-}$ are polynomial vector fields (and hence smooth, including in the critical temperature parameter), a
 sufficiently small translation  ensures  that  $Q^s_-$ is a stable node for the  flow $\psi_-$. We note there is a saddle $Q^u_-$ for $\psi_{-}$, near $Q^u_+$, as well.
 As $T_{cN}$ decreases to $-10^\circ$C, $Q^s_-\rw Q^s_+$ and $Q^u_-\rw Q^u_+$.

We note the behavior of the albedo lines for the flow exhibits an asymmetry when $T_{cN}\not= T_{cS}$.  As can be gleaned from Figure \ref{FIG-Y-structures}, when $T_{cN}=-5^\circ$C, $\eta_N(t)\rw 0.795$, a larger ice cap than in the case $\Tcn=-10^\circ$C. Of interest is the fact the Southern Hemisphere albedo line also moves to a larger (asymmetric)  ice cap position ($(\eta_S)^s_-=-0.907$), relative to its stable position when $T_{cN}=-10^\circ$C ($(\eta_S)^s_+=-0.955$). That is, the coupling of $\eta_S$ and $\eta_N$ provided by the $w$-equation in \eqref{system-NS} furnishes a linkage between the Northern and Southern Hemispheres: a different stable $\eta_N$ position yields a different stable $\eta_S$ position, even though $T_{cS}$ remains constant at $-10^\circ$C.

We plot the evolution of the albedo lines starting with large initial ice caps ($\eta_N=-\eta_S=0.5$) for system \eqref{system-NS} with $T_{cS}=-10^\circ$C fixed and various $T_{cN}$-values in Figure \ref{FIG-asym-time-series}.  
Similar behavior occurs if the Northern Hemisphere critical temperature is left at $T_{cN}=-10^\circ$C and $T_{cS}$ is increased.

\begin{figure}
\begin{center}
\includegraphics[width=6in,trim = 1.6in 7.5in 1in  1in, clip]{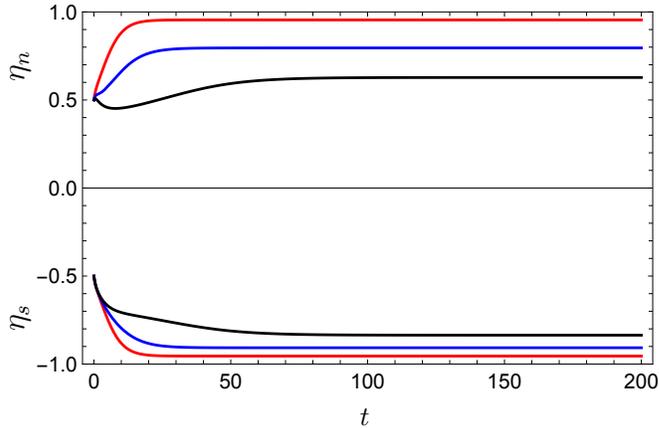}\\
\caption{{\small The evolution of the albedo lines under the flows $\psi_{\pm}$ with $\Tcs=-10^\circ$C. {\em Red}: $\Tcn=-10^\circ$C. \ {\em Blue}:  $\Tcn=-5^\circ$C. \ {\em Black}:  $\Tcn=-2^\circ$C.
}} 
\label{FIG-asym-time-series}
 \end{center}
 
\end{figure}

\section{Mass-Balance Can Drive Synchronous Global Glacial Cycles}
\label{Section-Mass-Balance}

In this section we incorporate the glacial cycle model introduced in \cite{wwhm} into the Northern Hemisphere of our global temperature, two albedo line model. The motivation for this model enhancement stems in part from the glacial cycle theory of M. Milankovitch, which asserts that changes in Northern Hemisphere high latitude insolation, due to variations in Earth's orbital elements over long time scales, comprise the principle forcing mechanism of the glacial-interglacial cycles  \cite{hays,milank, raylisnic,uemura}.

The glacial cycle model discussed below exhibits a threshold behavior, ``flip-flopping" between glacial advance and retreat based on a conceptual ice sheet mass balance equation. For more detailed background and motivation for this aspect of the model, the reader is referred to \cite{wwhm}.

\subsection{Mass balance flip-flop}

We begin by summarizing the process of adding a conceptual mass-balance variable $\xi_N$  in an effort to  model the accumulation and ablation of the Northern Hemisphere glaciers, as presented in \cite{wwhm}.  

Let $\xi_N$ denote the latitude of the edge of the Northern Hemisphere glaciers.  While the evolution of the ice edge $\xi_N$ is  driven in the abstract by a mass balance principle,  we
do not explicitly consider ice volume and mass here.

To construct the equations governing $\xi_N$ during glacial or interglacial periods, we assume snow is accumulating between $\eta_N$ and the north pole at a (dimensionless) rate $a$, while ablation
occurs between $\xi_N$ and $\eta_N$ at a (dimensionless) rate $b$. We note accumulation and ablation of ice play an important role in the advance, retreat, and size of a glacier (see, e.g.,  \cite{weertman}).  In particular, increased ablation rates when the glacier is retreating are key to obtaining the rapid interglacial retreats that are present in paleoclimate records \cite{abe-ouchi}. In this model it is the reduced albedo of the region between $\xi_N$ and $\eta_N$ due to factors such as aging snow \cite{gallee}, superglacial forest growth \cite{wright}, and dust loading \cite{peltier} that contributes to the increased ablation rate during glacial retreats.

We first define a critical ablation rate $b$.  Conceptually, the equation
\begin{equation}
b(\eta_N-\xi_N)=a(1-\eta_N)
\end{equation}
defines the Northern Hemisphere albedo- and ice-edge latitudes where ablation (left hand side) and accumulation (right hand side) are equal.  Rearranging this equation allows us to see that if
\begin{equation}
\xi_N>\left(1+\frac{a}{b}\right)\eta_N-\frac{a}{b}
\end{equation}
then the ablation $b(\eta_N-\xi_N)$ will be less than accumulation $a(1-\eta_N)$ and we should be in a glacial period (with the ice edge advancing).  In a glacial period the ablation is less than the critical ablation rate, so we let $b_-<b$ and set
\begin{equation}
    \dot\xi_N=\epsilon(b_-(\eta_N-\xi_N)-a(1-\eta_N)),\text{ when }\xi_N>\left(1+\frac{a}{b}\right)\eta_N-\frac{a}{b},
\end{equation}
with $\eps>0.$ 
On the other hand if
\begin{equation}
\xi_N<\left(1+\frac{a}{b}\right)\eta_N-\frac{a}{b}
\end{equation}
then the ablation $b(\eta_N-\xi_N)$ will be greater than accumulation $a(1-\eta_N)$ and we should be in an interglacial period with a large ablation rate and the ice edge retreating.  We let $b_+>b$ and set
\begin{equation}
    \dot\xi_N=\epsilon(b_+(\eta_N-\xi_N)-a(1-\eta_N)),\text{ when }\xi_N<\left(1+\frac{a}{b}\right)\eta_N-\frac{a}{b}.
\end{equation}
When $\xi_N - \left(\left(1+\frac{a}{b}\right)\eta_N-\frac{a}{b}\right)$ passes through 0, the system flips from one with a relatively low ablation rate to one with a relatively high ablation rate, or vice versa.

While fixing the critical temperature $T_{cS}=-10^\circ$C at $\eta_S$, we allow for different critical temperatures    at $\eta_N$ during the advance ($T^-_{cN}$) and retreat ($T^+_{cN}$) of the Northern Hemisphere glaciers, as intimated in Section \ref{Section-Symmetry}. We choose $T^-_{cN}>T^+_{cN}$ as  in \cite{wwhm}.

We are thus lead to consider the following $(w,\eta_S,\eta_N,\xi_N)$-system, one having discontinuities on a  hyperplane corresponding to points at which the Northern Hemisphere ice sheet mass balance equals zero.

\subsection{The full system: Southern and northern albedo lines with mass-balance flip-flop in the Northern Hemisphere}

The $w$- and $\eta_S$-equations remain as in system \eqref{system}, while the flip-flop mechanism   described above is placed  in the   Northern Hemisphere.  The system then assumes the form
\begin{subequations}\label{Nflipflop}
\begin{align}
\dot{w}&=-\tx{\frac{B}{R}}\left(w-F(\eta_S,\eta_N)\right)\label{NflipflopA} \\
\dot{\eta_S} &=-\rho(w-G(\eta_S))\label{NflipflopB} \\ 
\dot{\eta_N} &=\rho(w-H_\pm(\eta_N))\label{NflipflopC} \\ 
\dot{\xi_N} &=\eps(b_\pm(\eta_N-\xi_N)-a(1-\eta_N)), \label{NflipflopD}
\end{align}
\end{subequations}
where $F$ and $G=G_S$ are as in \eqref{F} and \eqref{GS}, respectively, and where  we set 
\begin{equation}\notag
H_+(\eta_N)=-L(1-\al_0)(s(\eta_N)-1)+T^+_{cN} \ \mbox{ \ and \ } \ H_-(\eta_N)=-L(1-\al_0)(s(\eta_N)-1)+T^-_{cN}.
\end{equation}
The use of the subscript `+' indicates $\xi_N<\left(1+\frac{a}{b}\right)\eta_N-\frac{a}{b}$, so that the ice sheet is retreating in the  Northern Hemisphere.  The subscript `-' indicates $\xi_N>\left(1+\frac{a}{b}\right)\eta_N-\frac{a}{b}$, with the Northern Hemisphere glaciers advancing equatorward in this regime.

The state space for \eqref{Nflipflop} is 
\begin{equation}\notag
{\mathcal B}=\{  (w,\eta_S,\eta_N,\xi_N) : w\in\R, \ \eta_S,\eta_N,\xi_N \in [-1,1],\eta_S\leq \eta_N  \}.
\end{equation}
We note there will be no consideration of the dynamics on the boundary of $\mathcal{B}$ in this paper; the results and analysis to follow pertain to an invariant subset of $\mathcal{B}$ in which $-1<\eta_S<\eta_N<1$.

Recall we are assuming the critical temperatures and ablation rates satisfy  $T^+_{cN}<T^-_{cN}$ \ and \ $b_-<b<b_+$, respectively. Finally, while the analysis   in this section holds for any $M$-value and appropriately chosen parameters, we continue to set $M=1$. Thus,  $G(\eta_S)$ and $H_\pm(\eta_N)$ are each quadratic polynomials, and $F(\eta_S,\eta_N)$ is the difference of a cubic polynomial in $\eta_N$ and a cubic polynomial in $\eta_S$. 

Due to the presence of discontinuities induced by the switching mechanism from Northern Hemisphere glacial advance to retreat (and vice versa) discussed above, we analyze system \eqref{Nflipflop} as a {\em Filippov flow.}

To define the Filippov flow associated with system \eqref{Nflipflop}, we begin by letting
\begin{equation}\label{hsplit}
h:{\mathcal B}\rw \R, \ h(w,\eta_S,\eta_N,\xi_N)=b(\eta_N-\xi_N)-a(1-\eta_N)=(a+b)\eta_N-b\xi_N-a.
\end{equation}
The {\em switching manifold} \cite{dibernardo}, consisting of points in $\mathcal{B}$ at which the critical mass balance $b(\eta_N-\xi_N)-a(1-\eta_N)$ equals 0,  is the hyperplane
\begin{align}\label{switch}
\Sigma &=\{(w,\eta_S,\eta_N,\xi_N) : h(w,\eta_S,\eta_N,\xi_N)=0\}\\\notag
&=\{(w,\eta_S,\eta_N,\xi_N) : \xi_N=(1+\tx{\frac{a}{b}})\eta_N-\tx{\frac{a}{b}}=\gamma(\eta_N)\}.\notag
\end{align}
The system is retreating toward an interglacial period when in the region 
\begin{equation}\label{Splus}
S_+=\{(w,\eta_S,\eta_N,\xi_N) :  h(w,\eta_S,\eta_N,\xi_N)>0\},
\end{equation}
and   advancing to a glacial period when in 
\begin{equation}\label{Sminus}
S_-=\{(w,\eta_S,\eta_N,\xi_N) :  h(w,\eta_S,\eta_N,\xi_N)<0\}.
\end{equation}

Let ${\bf X}_+$ denote system \eqref{Nflipflop} when choosing $H_+$ and $b_+$, and let ${\bf X}_-$ denote system \eqref{Nflipflop} when choosing $H_-$ and $b_-$.  
For ${\bf v}=(w,\eta_S,\eta_N,\xi_N)\in \mathcal{B},$ we then consider  the differential inclusion
\begin{equation}\label{inclusion}
\dot{\bf v}\in {\bf X}({\bf v})=
\begin{cases}
{\bf X}_-({\bf v}), &  {\bf v}\in S_-\\ 
\{ (1-p){\bf X}_-({\bf v})+p{\bf X}_+({\bf v}) : p\in [0,1]\}, & {\bf v}\in\Sigma \\ 
{\bf X}_+({\bf v}), &  {\bf v}\in S_+ . 
\end{cases}
\end{equation}
Note each of ${\bf X}_\pm$ is smooth on $S_\pm.$ While in $S_-$, solutions are unique with flow $\phi_-({\bf v},t)$ corresponding to system  $\dot{\bf v}={\bf X}_-({\bf v})$. Similarly, solutions in 
$S_+$ are unique with flow $\phi_+({\bf v},t)$ given by system $\dot{\bf v}={\bf X}_+({\bf v})$. For ${\bf v}\in\Sigma, \ \dot{\bf v}$ must lie in the closed convex hull of the two vectors ${\bf X}_-({\bf v})$ and ${\bf X}_+({\bf v})$.

A solution to \eqref{inclusion} {\em in the sense of Filippov} is an absolutely continuous function ${\bf v}(t)$ satisfying $\dot{\bf v}\in {\bf X}({\bf v})$ for almost all $t$. (Note $\dot{\bf v}(t)$ is not defined at times for which ${\bf v}(t)$ arrives at or leaves $\Sigma$.) Given that ${\bf X}_\pm$ are continuous on $S_\pm\cup \Sigma$, the set-valued map ${\bf X}({\bf v})$ is upper semi-continuous, and closed, convex and bounded for all ${\bf v}\in \mathcal{B}$ and $t\in\R$. This implies that for each ${\bf v}_0\in \mbox{Int}(\mathcal{B})$ there is a solution ${\bf v}(t)$ to differential inclusion \eqref{inclusion} in the sense of Filippov, defined  on an interval $[0,t_f]$, with ${\bf v}(0)={\bf v}_0$   \cite{leine}.

\subsection{Regular and  virtual equilibria}

As equations \eqref{NflipflopA}--\eqref{NflipflopC} decouple from \eqref{NflipflopD}, we first note that the vector fields   corresponding to  \eqref{NflipflopA}--\eqref{NflipflopC} are precisely  the vector fields ${\bf Y}_\pm$ from Section \ref{Section-Symmetry} with associated flows $\psi_\pm=\psi_{\pm}((w,\eta_S,\eta_N),t)$. 

Let $W(Q^s_+)$ denote the $\psi_+$-stable set of $Q^s_+$, and let $W(Q^s_-)$ denote the $\psi_-$-stable set of $Q^s_-$, noting that each stable set is a subset of $\mathcal{B}^\pr$ with interior.
By smoothness of the vector fields ${\bf Y}_\pm$ (each smooth in the critical temperature as well), we
 choose $\Tcnm$ close enough to $\Tcnp$ to ensure that
\begin{equation}\label{3Dstablesets}
Q^s_-\in W(Q^s_+) \ \mbox{ and } \ 
 Q^s_+\in W(Q^s_-),
 \end{equation}
the motivation for which will become apparent below. Numerical investigations indicate that conditions \eqref{3Dstablesets} hold for $\Tcnm$ as large as $-1^\circ$C. We also note \eqref{3Dstablesets} holds for all $\eps>0$, where $\eps$ governs the rate of the mass balance response to albedo change as in equation \eqref{NflipflopD}.

Returning to the vector fields ${\bf X}_\pm$ associated with the full system \eqref{Nflipflop},   ${\bf X}_+$ then admits two equilibria in ${\mathcal B}$
\begin{align*}
P^u_+&=\left(w^u_+,(\eta_S)^u_+,(\eta_N)^u_+,(1+\tx{\frac{a}{b_+}})(\eta_N)^u_+-\tx{\frac{a}{b_+}}\right) \ \mbox{ and } \\ 
P^s_+&=\left(w^s_+,(\eta_S)^s_+,(\eta_N)^s_+,(1+\tx{\frac{a}{b_+}})(\eta_N)^s_+-\tx{\frac{a}{b_+}}\right).
\end{align*}
As the fourth column of the Jacobian matrix $J{\bf X}_+$ is $[0 \ 0 \ 0 \ \tx{-}\eps b_+]^T$, we conclude $P^s_+$ is a stable node for the retreating flow $\phi_+$ for all $\eps>0$ (while $P^u_+$ is a saddle with 3-dimensional stable manifold).

We would like to know which side of the switching manifold $\Sigma$ the equilibrium $P_{+}^s$ lies in. A computation yields 
\begin{equation}\notag
h(P^s_+)= a \left(1-(\eta_N)^s_+\right)\left(\tx{\frac{b}{b_+}}-1\right)<0
\end{equation}
due to our assumption $b_+>b$, implying $P^s_+\in S_-$ (see equation \eqref{Sminus}). Thus $\phi_+$-trajectories are unable to converge to the stable node $P^s_+$ as they must first cross the switching manifold $\Sigma.$ Such an equilibrium point for a discontinuous vector field is known as a virtual equilibrium point \cite{dibernardo}, as defined in the introduction.

In a similar fashion, and recalling our choice of the parameter $\Tcnm$ as discussed above, the vector field ${\bf X}_-$ admits two equilibria
\begin{align*}
P^u_-&=\left(w^u_-,(\eta_S)^u_-,(\eta_N)^u_-,(1+\tx{\frac{a}{b_-}})(\eta_N)^u_--\tx{\frac{a}{b_-}}\right) \ \mbox{ and } \\ 
P^s_-&=\left(w^s_-,(\eta_S)^s_-,(\eta_N)^s_-,(1+\tx{\frac{a}{b_-}})(\eta_N)^s_--\tx{\frac{a}{b_-}}\right),
\end{align*}
with $P^s_-$ a stable node for all $\eps>0$ (and $P^u_-$ a saddle having 3-dimensional stable manifold). Importantly,
\begin{equation}\notag
h(P^s_-)= a \left(1-(\eta_N)^s_-\right)\left(\tx{\frac{b}{b_-}}-1\right)>0
\end{equation}
since $b_-<b$. Hence $P^s_-$ is also a virtual equilibrium point for \eqref{inclusion} as $P^s_-\in S_+$ (see equation \eqref{Splus}). 

Let $W(P^s_+)$ denote the stable set of $P^s_+$ under the retreating flow $\phi_+$, and let $W(P^s_-)$ denote the stable set of $P^s_-$ under the advancing flow $\phi_-$. Recall we are choosing $\Tcnm$ close enough to $\Tcnp$ to ensure conditions \eqref{3Dstablesets}, that $Q_{\pm}^s$ were in each other's stable sets under the three-dimensional flows $\psi_\pm$. Given the decoupling of equations \eqref{NflipflopA}--\eqref{NflipflopC} from \eqref{NflipflopD}, along with the linear nature of equation \eqref{NflipflopD}, note \eqref{3Dstablesets} implies 
\begin{equation}\label{4Dstablesets}
P^s_-\in W(P^s_+) \ \mbox {and } \ P^s_+\in W(P^s_-).
\end{equation}
This observation, which holds for all $\eps>0$,  will play a key role in elucidating the flip-flop behavior of our model.

\subsection{Trajectories intersecting the switching manifold}

We begin by determining where on the 3-dimensional switching manifold $\Sigma$ the vector fields ${\bf X}_\pm$ are tangent, as such submanifolds may bound sliding regions \cite{leine}. To that end, $\Sigma$ is a hyperplane with normal vector
${\bf N}= [0 \  \ 0 \  \ 1+\tx{\frac{a}{b}} \ \ \tx{-}1 ]^T$. \ For ${\bf v}\in\Sigma$, a computation yields ${\bf X}_+ \perp {\bf N}$ if and only if
\begin{equation}\label{hplus}
w=H_+(\eta_N)+\frac{a\eps (1-\eta_N)(b_+-b)}{\rho (a+b)}=h_+(\eta_N).
\end{equation} 
Thus, ${\bf X}_+$ is tangent to $\Sigma$ at points contained in the set 
\begin{equation}\label{omega+}
\Omega_+=\{(h_+(\eta_N), \eta_S,\eta_N, \gamma(\eta_N)) : \eta_s,\eta_N\in [-1,1]\},
\end{equation}
a 2-dimensional submanifold of $\Sigma$ (recall $\gamma(\eta_N)$ is as defined in \eqref{switch}).  In a similar fashion, one finds ${\bf X}_- \perp {\bf N}$ at ${\bf v}\in\Sigma$ if and only if
\begin{equation}\label{omega-}
{\bf v}\in \Omega_-=\{(h_-(\eta_N), \eta_S,\eta_N, \gamma(\eta_N)) : \eta_S,\eta_N\in [-1,1]\},
\end{equation}
where 
\begin{equation}\label{hminus}
h_-(\eta_N)= H_-(\eta_N)+\frac{a\eps (1-\eta_N)(b_--b)}{\rho (a+b)}.
\end{equation}

We consider the case in which the surfaces of tangency $\Omega_\pm$ on the switching manifold $\Sigma$ do not intersect in $\mathcal{B}$. A tedious and straightforward calculation reveals that if the time constant $\eps$ in \eqref{NflipflopD} satisfies
\begin{equation}\label{epsbound}
\eps<\frac{(\Tcnm-\Tcnp)\rho (a+b)}{2a(b_+-b_-)},
\end{equation}
then $h_+(\eta_N)<h_-(\eta_N)$ for $\eta_N\in [-1,1].$ We assume $\eps$ satisfies \eqref{epsbound} in all that follows.

Having identified the sets of tangencies on either side of the switching manifold, we must now determine where the vector fields $\mathbf{X}_\pm$ point into the switching manifold $\Sigma$ and where they point away. Via further computations, we see for ${\bf v}=(w,\eta_S,\eta_N,\gamma(\eta_N))\in\Sigma$,
\begin{itemize}
\item[(i)] ${\bf X}_+({\bf v}) \dotp {\bf N}>0$ if $w>h_+(\eta_N)$, so that ${\bf X}_+({\bf v})$ points into $S_+$ if $w>h_+(\eta_N)$,
\item[(ii)] ${\bf X}_+({\bf v}) \dotp {\bf N}<0$ if $w<h_+(\eta_N)$, so that ${\bf X}_+({\bf v})$ points into $S_-$ if $w<h_+(\eta_N)$,
\item[(iii)] ${\bf X}_-({\bf v}) \dotp {\bf N}>0$ if $w>h_-(\eta_N)$, so that ${\bf X}_-({\bf v})$ points into $S_+$ if $w>h_-(\eta_N)$, and
\item[(iv)] ${\bf X}_-({\bf v}) \dotp {\bf N}<0$ if $w<h_-(\eta_N)$, so that ${\bf X}_-({\bf v})$ points into $S_-$ if $w<h_-(\eta_N)$.
\end{itemize}
In particular, a $\phi_+$-trajectory that intersects $\Sigma$ at a point for which $w<h_+(\eta_N)<h_-(\eta_N)$ passes transversally into $S_-$ following the Filippov convention, and continues in $S_-$ under the flow $\phi_-$. The subset  $\Sigma_+\subset \Sigma$ defined by
\begin{equation}\label{sigmaplus}
\Sigma_+=\{(w,\eta_S,\eta_N,\gamma(\eta_N)) : w<h_+(\eta_N), \ \eta_S,\eta_N\in [-1,1]\}
\end{equation}
is therefore known as a {\em crossing region} for the Filippov flow \cite{leine}. Similarly, a $\phi_-$-trajectory that intersects $\Sigma$ at a point in the set
\begin{equation}\label{sigmaminus}
\Sigma_-=\{(w,\eta_S,\eta_N,\gamma(\eta_N)) : w>h_-(\eta_N), \ \eta_S,\eta_N\in [-1,1]\}
\end{equation}
passes transversally into $S_+$   and continues by following    the flow $\phi_+$. In this fashion $\Sigma_-\subset\Sigma$ is also a crossing region for the Filippov flow. We note solutions to system \eqref{inclusion} that pass through $\Sigma_\pm$ are unique, though not differentiable at points of intersection with $\Sigma_\pm$.

Finally, consider the subset of the switching manifold defined by
\begin{equation}\label{slide}
\Sigma^{\mbox{\scriptsize  SL}}=\{(w,\eta_S,\eta_N,\gamma(\eta_N)) : h_+(\eta_N)<w<h_-(\eta_N), \ \eta_S,\eta_N\in [-1,1]\}.
\end{equation}
Note  ${\bf X}_+$ points into $S_+$ and  ${\bf X}_-$ points into $S_-$ at all points in $\Sigma^{\mbox{\scriptsize  SL}}$. The subset $\Sigma^{\mbox{\scriptsize SL}}$ of the switching manifold $\Sigma$ is therefore a {\em  repelling sliding region} \cite{leine}; Filippov's approach does not provide for unique solutions ${\bf v}(t)$ in forward time if ${\bf v}(0)\in 
\Sigma^{\mbox{\scriptsize SL}}$ \cite{fil}.  Notice that for $\eps$ chosen to satisfy \eqref{epsbound}, the repelling sliding region $\Sigma^{\mbox{\scriptsize SL}}$ sits between the tangency sets $\Omega_+$ and $\Omega_-$, thereby separating $\Sigma_+$ and $\Sigma_-$, throughout $\Sigma$.   As neither advancing nor retreating trajectories approach $\Sigma^{\mbox{\scriptsize  SL}}$ in forward time, the repelling sliding region will play no role in the analysis to come. Projections of $\Sigma$ and its subsets described above into $(w,\eta_N,\xi_N)$-space are plotted in Figure \ref {FIG-Sigma-proj}.

We see that  a trajectory for system \eqref{inclusion} with initial condition ${\bf v}(0)\in S_+ \cap W(P^s_+)$ will ``retreat" under the flow $\phi_+$, intersecting $\Sigma_+$ prior to approaching the virtual equilibrium $P^s_+$ and thereby switching to the ``advancing" flow $\phi_-$. With the parameters $b_\pm$ chosen appropriately (as discussed in the following section), this $\phi_-$-trajectory will intersect $\Sigma_-$ on its way to  approaching the virtual equilibrium $P^s_-$, thereby flipping back to the retreating  flow $\phi_+$. 

We now prove the  dynamic described above is capable of producing a unique (nonsmooth) attracting periodic orbit that, in terms of the model,   represents  the  glacial-interglacial cycles, entirely a consequence of the flip-flop in the  Northern Hemisphere. 

\begin{figure}
\begin{center}
\includegraphics[width=5.6in,trim = 1.6in 6.6in 1in  1in, clip]{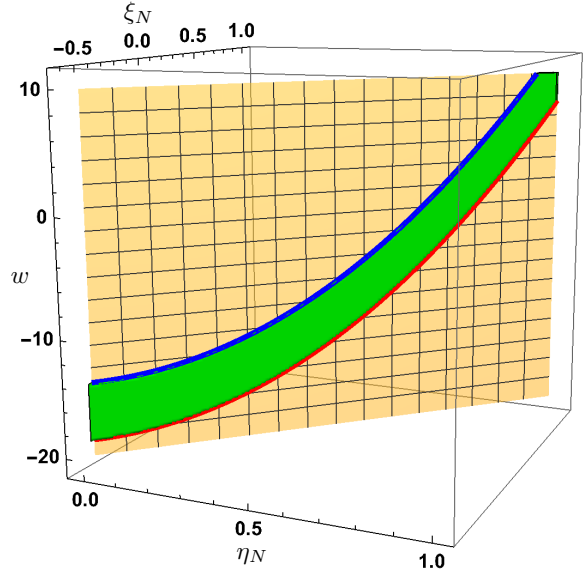}\\
\caption{{\small Projections into $(w,\eta_N,\xi_N)$-space of the switching manifold $\Sigma$ (gold), surfaces of tangency $\Omega_+$ (red) and $\Omega_-$ (blue), and the repelling sliding region $\Sigma^{\mbox{\scriptsize SL}}$ (green). $\Sigma_-$ projects to the region above the blue curve, while  $\Sigma_+$   projects to the region below the red curve.  
}}
\label{FIG-Sigma-proj}
 \end{center}
\end{figure}

\subsection{A  return map for the Filippov flow}

In constructing the return map, it is instructive to first consider the case in which $\eps=0$. Note when  $\eps=0$ the retreating flow $\phi_+$ has an attracting line $\ell_+$ of equilibrium points. That is, if ${\bf v}(0)\in S_+\cap W(P^s_+), \ \phi_+({\bf v}(0),t)\rw (w^s_+, (\eta_S)^s_+, (\eta_N)^s_+,\xi_N(0))\in\ell_+$ as $ t\rw\infty$. We remark that $\ell_+$ intersects $\Sigma$ at the point
\begin{equation}\label{Rplus}
R_+=(w^s_+, (\eta_S)^s_+, (\eta_N)^s_+, \gamma((\eta_N)^s_+)).
\end{equation}
Similarly, the advancing flow $\phi_-$ has an attracting  line $\ell_-$ of equilibria when $\eps=0$; if 
 ${\bf v}(0)\in S_-\cap W(P^s_-), \ \phi_-({\bf v}(0),t)\rw (w^s_-, (\eta_S)^s_-, (\eta_N)^s_-,\xi_N(0))\in\ell_-$ as $ t\rw\infty$. The line $\ell_-$ intersects $\Sigma$ at the point
\begin{equation}\label{Rminus}
R_-=(w^s_-, (\eta_S)^s_-, (\eta_N)^s_-, \gamma((\eta_N)^s_-)).
\end{equation}
These points of intersection will help us determine where trajectories are crossing the switching manifold.  If we have $R_-\in W(P^s_+)$ and if $R_+\in W(P^s_-)$, the existence of a periodic orbit of  Filippov system \eqref{inclusion} would seem plausible.

Now suppose $\eps>0$ is much smaller then the time constant $\rho$ in \eqref{Nflipflop}. The $\phi_+$-trajectory of a point ${\bf v}(0)\in S_+\cap W(P^s_+)$ will first approach the line $\ell_+$ with $\xi_N(t)$ varying little from $\xi_N(0)$, and then follow $\ell_+$ toward  the switching manifold, intersecting $\Sigma$ at a point near $R_+$.  Absent the presence of the switching manifold, this dynamic is reminiscent  of problems addressed by geometric singular perturbation theory for smooth dynamical systems having multiple time scales \cite{jones}.

Note that as $b_+$ decreases to $b,$ the $\phi_+$-stable node  $P^s_+$ approaches $R_+$ because only the fourth coordinate of $P_{+}^s$ varies with $b_+$.  Indeed, the fourth coordinate of $P_{+}^s$ is  $(\xi_N)_{+}^s=(1+\frac{a}{b_+})(\eta_N)_{+}^s - \frac{a}{b_+}$ which limits to $\gamma((\eta_N)_{+}^s)$ as $b_+\searrow b$. Hence we will assume $b_+$ is chosen to ensure that the point $R_+$ is in the stable set of $P_{+}^s$ under the retreating flow $\phi_+$, $ W(P^s_+)$. Recall that $P^s_+$ is also in the $W(P^s_-)$, the stable set of $P_{-}^s$ under the advancing flow $\phi_-$ \eqref{4Dstablesets},  which implies the existence of a neighborhood $U$ of $P^s_+$ with $U\subset W(P^s_-)$. We then additionally assume  $b_+$ is close enough to $b$ to ensure $R_+\in U$, so that $R_+\in W(P^s_-)\cap W(P^s_+)$, an inclusion that holds for all $\eps>0.$

In a similar vein, the $\phi_-$-trajectory of a point ${\bf v}(0)\in S_-\cap W(P^s_-)$ will first approach the line $\ell_-$ with $\xi_N(t)$ remaining roughly constant, and then follow $\ell_-$ toward  the switching manifold, intersecting $\Sigma$ at a point near $R_-$.   
As $b_-\nearrow b,$ the $\phi_-$-stable node  $P^s_-$ approaches $R_-$ (and so we assume $R_-\in W(P^s_-)$). As  $P^s_-\in W(P^s_+)$ \eqref{4Dstablesets}, there then  exists  a neighborhood $U$ of $P^s_-$ with $U\subset W(P^s_+)$.  Choosing $b_-$ sufficiently close to $b$ then ensures that  $R_-\in W(P^s_+)\cap W(P^s_-)$, which we again note  holds for all $\eps>0$.

The above choices of parameters $b_+$ and $b_-$ (and of $\Tcnm$ previously) now allow for the construction of a (nonsmooth) return map for the Filippov flow \eqref{inclusion} as follows. 

We begin by noting that $R_+$ is in the crossing region $\Sigma_+$ (where trajectories cross from $S_+$ to $S_-$) because $w^s_+=H_+((\eta_N)^s_+)<h_+((\eta_N)^s_+)$ by \eqref{hplus}. As we have just seen that $R_+\in W(P^s_-),$ we can pick $\delta_1>0$ such that
\begin{equation}\label{Vplus}
V_+=\ov{B_{\delta_1}}(R_+)\cap\Sigma\subset W(P^s_-)\cap \Sigma_+.
\end{equation}
Recalling $P^s_-$ is a virtual equilibrium point for the advancing flow $\phi_-$, for any ${\bf v}\in V_+$ and for any $\eps>0$  there exists a time $t=t({\bf v},\eps)$ such that $\phi_-({\bf v}, t({\bf v},\eps))$ reaches the crossing region  $\Sigma_-$ (where trajectories cross from $S_-$ to $S_+$). We note $t({\bf v},\eps)\rw\infty$ as $\eps\rw 0$ for future reference. 
We may then define a continuous mapping, for any $\eps>0$, given by
\begin{equation}\label{returnminus}
r^\eps_- :V_+\rw \Sigma_-, \ r^\eps_-({\bf v})=\phi_-({\bf v}, t({\bf v},\eps)).
\end{equation}

That $R_-\in \Sigma_-$ follows from the fact that $w^s_-=H_-((\eta_N)^s_-)>h_-((\eta_N)^s_-)$ by \eqref{hminus}.  Recalling $R_-\in W(P^s_+),$ we can pick $\delta_2>0$ such that
\begin{equation}\label{Vminus}
V_-=\ov{B_{\delta_2}}(R_-)\cap\Sigma\subset W(P^s_+)\cap \Sigma_-.
\end{equation}
Noting $P^s_+$ is a virtual equilibrium point for the retreating flow $\phi_+$, for any ${\bf v}\in V_-$ and for any $\eps>0$  there exists $t=t({\bf v},\eps)$ such that $\phi_+({\bf v}, t({\bf v},\eps))\in \Sigma_+$. Hence for any $\eps>0$, we define the continuous mapping
\begin{equation}\label{returnplus}
r^\eps_+ :V_-\rw \Sigma_+, \ r^\eps_+({\bf v})=\phi_+({\bf v}, t({\bf v},\eps)).
\end{equation}
We are now in a position to prove there exists $\eps>0$ such that $r^\eps=r^\eps_+\circ r^\eps_- :V_+\rw V_+$ is a contraction map.

\subsection{Existence of an attracting limit cycle}

\begin{prop}  (a) Given $c\in (0,1)$, there exists $\hat{\eps}$ such that for all $\eps\leq\hat{\eps} $ and for all ${\bf v}_1, {\bf v}_2\in V_+,$
\begin{equation}\notag
\|r^\eps_-({\bf v}_2)-r^\eps_-({\bf v}_1) \|\leq c\|{\bf v}_2-{\bf v}_1\|.
\end{equation}
(b) Given $c\in (0,1)$, there exists $\hat{\eps}$ such that for all $\eps\leq\hat{\eps} $ and for all ${\bf v}_1, {\bf v}_2\in V_-,$
\begin{equation}\notag
\|r^\eps_+({\bf v}_2)-r^\eps_+({\bf v}_1) \|\leq c\|{\bf v}_2-{\bf v}_1\|.
\end{equation}
\label{prop1}
\end{prop}

\begin{proof}  
We prove case (a). In this proof, for ease of notation, we set $x=\eta_S, y=\eta_N$ and $z=\xi_N$. Relying on the fact equation \eqref{NflipflopD} decouples from equations \eqref{NflipflopA}--\eqref{NflipflopC}, the proof is in spirit  analogous to the proof of Proposition 5.4 in \cite{wwhm}; we include it here for completeness.

Let $c\in (0,1)$, and let ${\bf v}=(w_0,x_0,y_0,\gamma(y_0)) \in V_+$. Recall that by design, under the advancing flow corresponding to equations \eqref{NflipflopA}--\eqref{NflipflopC} we have $\psi_-((w_0,x_0,y_0),t)\rw Q^s_-$ as $t\rw\infty$. Since $V_+$ is the intersection of a closed ball in $\R^4$ with the hyperplane $\Sigma$, $V_+$ is compact (as well as connected and convex). Thus the set $J=\{ (w,x,y) : (w,x,y,\gamma(y))\in V_+\}$ is a compact set which, coupled with the fact $J\subset W(Q^s_-)$, yields the existence of $T_1$ such that for all $t\geq T_1$ and for all ${\bf u}_1,  {\bf u}_2\in J, \ \|\psi_-({\bf u}_2)-\psi_-({\bf u}_1)\|\leq c\|{\bf u}_2-{\bf u}_1\|$.  

Given ${\bf v}\in V_+$, pick $\eps({\bf v})$ such that $t({\bf v}, \eps({\bf v}))>T_1$, where $t({\bf v}, \eps({\bf v}))$ is as in the definition of $r^\eps_- $ \eqref{returnminus}. By the continuity  of $\phi_-$ with respect to initial conditions and time, there exists $\delta({\bf v})>0$ so that ${\bf w}\in B_{\delta({\bf v})}({\bf v})\cap V_+$ implies $t({\bf w}, \eps({\bf v}))>T_1$ (where $r^{\eps({\bf v})}_-({\bf w})\in \Sigma_-$). We note for any $\eps\leq \eps({\bf v}), \ t({\bf w}, \eps)>T_1$. 

In this fashion we arrive at an open covering 
\begin{equation}\notag
V_+\subset \bigcup_{{\bf v}\in V_+} B_{\delta({\bf v})}({\bf v})
\end{equation}
of the compact set $V_+$. Choose a finite subcover $\{ B_{\delta({\bf v}_n)}({\bf v}_n) : n=1, ... , N\}$, and let $\hat{\eps}=\min\{ \eps({\bf v}_n) : n=1, ... , N\}$. Then for any $\eps\leq\hat{\eps}$ and for all ${\bf v}\in V_+, \ t({\bf v},\eps)>T_1$.

Suppose $\eps\leq\hat{\eps}$, and let ${\bf v}_1=(w_1,x_1,y_1, \gamma(y_1))$ and $  {\bf v}_2=(w_2,x_2,y_2, \gamma(y_2))$ be elements in $V_+$. Set ${\bf u}_1=(w_1,x_1,y_1)$ and $
{\bf u}_2=(w_2,x_2,y_2).$

Let $r^\eps_-({\bf v}_1)=(w^\pr_1, x^\pr_1, y^\pr_1, \gamma(y^\pr_1)), \ r^\eps_-({\bf v}_2)=(w^\pr_2, x^\pr_2, y^\pr_2, \gamma(y^\pr_2)),  \ 
{\bf u}^\pr_1=(w^\pr_1, x^\pr_1, y^\pr_1)   $ and ${\bf u}^\pr_2=(w^\pr_2, x^\pr_2, y^\pr_2)$. Note that by our choice of $\eps, \ \|{\bf u}^\pr_2-{\bf u}^\pr_1\|\leq c \|{\bf u}_2-{\bf u}_1\|.$
\ We then have
\begin{align}\notag
\|r^\eps_-({\bf v_2})-r^\eps_-({\bf v_1})\|^2 & = \|{\bf u}^\pr_2-{\bf u}^\pr_1\|^2+(\gamma(y^\pr_2)-\gamma(y^\pr_1))^2\\\notag
& = \|{\bf u}^\pr_2-{\bf u}^\pr_1\|^2+ (1+\textstyle{\frac{a}{b}})^2 (y^\pr_2- y^\pr_1)^2\\\notag
 & \leq c^2\|{\bf u}_2-{\bf u}_1\|^2+(1+\textstyle{\frac{a}{b}})^2 c^2 (y_2-y_1)^2\\\notag
&= c^2\|{\bf u}_2-{\bf u}_1\|^2+c^2(\gamma(y_2)-\gamma(y_1))^2\\\notag
&=c^2\|{\bf v}_2-{\bf v}_1\|^2.
\end{align}
A similar argument can be given to prove that for $c\in (0,1), $ there exists $\hat{\eps}>0$ so that for all $\eps\leq \hat{\eps}, $ \ $r^\eps_+:V_-\rw \Sigma_+$ contracts distances by a factor of at most $c$.
\end{proof}

\begin{prop}  (a) There exists $\hat{\eps}>0$ such that for all $\eps\leq\hat{\eps} $, \ $r^\eps_- : V_+\rw V_-$.
\\ (b) There exists $\hat{\eps}>0$ such that for all $\eps\leq\hat{\eps} $, \ $r^\eps_+ : V_-\rw V_+$. 
\label{prop2}
\end{prop}

\begin{proof}  We prove case (a). Let ${\bf v}=(w,\eta_S,\eta_N,\gamma(\eta_N))_{t=0}\in V_+$. We again note that $$\psi_-((w(0),\eta_S(0),\eta_N(0)),t)=(w(t),\eta_S(t),\eta_N(t))\rw Q^s_-\text{ as }t\rw\infty.$$ Additionally using the fact $\gamma (\eta_N)$ is continuous, pick $T=T({\bf v})$ such that for all $t\geq T,$
\begin{equation}\notag
\|(w(t),\eta_S(t),\eta_N(t),\gamma(\eta_N(t)))-R_-\|<\tx{\frac{1}{2}}\delta_2,
\end{equation} 
where $\delta_2$ is as in \eqref{Vminus}. 
Choose $\eps({\bf v})>0$ such that $t({\bf v}, \eps({\bf v}))>T$   (where $t({\bf v}, \eps({\bf v}))$ is as in \eqref{returnminus}). Then for all $\eps\leq \eps({\bf v}), \ \|r^\eps_-({\bf v})-R_-\|< \tx{\frac{1}{2}}\delta_2$.

Let $0<c<\text{min}\{1,\tx{\frac{1}{2}}\delta_2\}$, and pick $\hat{\eps}$ as in Proposition \ref{prop1}. For $\eps\leq \min\{\hat{\eps},\eps({\bf v})\}$ and for any ${\bf w}\in V_+$,
\begin{equation}\notag
\|r^\eps_-({\bf w})-R_-\|\leq \|r^\eps_-({\bf w}) - r^\eps_-({\bf v})\|+\|r^\eps_-({\bf v})-R_-\|<\tx{\frac{1}{2}}\delta_2+\tx{\frac{1}{2}}\delta_2=\delta_2.
\end{equation} 
Thus $r^\eps_-({\bf w})\in V_-$, and we conclude $r^\eps_-(V_+)\subset V_-$. 
\end{proof}

\vspace{0.01in}
\begin{thm}  With other parameters as in Table 1, choose $\Tcnm$- and $b_\pm$-values, respectively, so that for any $\eps>0$ \\
\hspace*{0.2in}(i) $P^s_-\in W(P^s_+)$, \ $P^s_+\in W(P^s_-)$, and\\
\hspace*{0.2in}(ii) $R_+, R_-\in W(P^s_-)\cap W(P^s_+)$,\\
 as discussed above. Then system \eqref{inclusion} admits a unique attracting limit cycle for sufficiently small $\eps>0$.
\label{cycle}
\end{thm}

\begin{proof} Given $c\in (0,1)$, choose $\hat{\eps}_1$ and $\hat{\eps}_2$ as in cases (a) and (b), respectively, in Proposition \ref{prop1}. Also choose $\hat{\eps}_3$ and $\hat{\eps}_4$ as in cases (a) and (b), respectively, in Proposition \ref{prop2}. Let $\eps\leq \min\{\hat{\eps}_n : n=1, ... , 4\}$. Then since   $r^\eps_- :V_+\rw V_-$ and $r^\eps_+:V_-\rw V_+$, we can define the return map
\begin{equation}\notag
r^\eps=r^\eps_+\circ r^\eps_- :V_+\rw V_+.
\end{equation}
As $r^\eps$ is additionally a contraction map with contraction factor $c^2$, $r^\eps$ has a unique  fixed point ${\bf v}^*\in V_+ $ to which all $r^\eps$-orbits converge. The Filippov trajectory that flows via $\phi_-$ from ${\bf v}^*$ to ${\bf w}=r^\eps_-({\bf v}^*)$, and from ${\bf w}$ back to ${\bf v}^*$ via $\phi_+$, is then an attracting (nonsmooth) limit cycle.
\end{proof}

\vspace{0.01in}\noindent
{\bf Remark.} Every trajectory of system \eqref{inclusion} that passes through $V_+$ (or through $V_-$) converges to the limit cycle $\Gamma$ provided by Theorem \ref{cycle}. We note that given any compact set $K\subset W(P^s_+) \cap S_+$, one can choose $\eps$ sufficiently small  to ensure  $K$ is contained in the stable set of $\Gamma$ under the Filippov flow.

\vspace{0.1in}

\section{Selected Numerical Results}
In Figure \ref{FIG-neg5-neg10} we plot the projection of the limit cycle $\Gamma$ into the three-dimensional $(\eta_S,\eta_N,\xi_N)$-space, along with the behavior of $-\eta_S, \eta_N$ and $\xi_N$ over time along $\Gamma$, in the case $\Tcnm=-5^\circ$C. We first note the sawtooth pattern evident in the evolution of each of the variables, with a rapid retreat into an interglacial period following a slower descent into a glacial age, as seen in the climate data over the past  1 million years \cite{raylisnic}.

\begin{figure}
\begin{center}
\begin{tikzpicture}
  \node (img1)  {\includegraphics[width=2.8in]{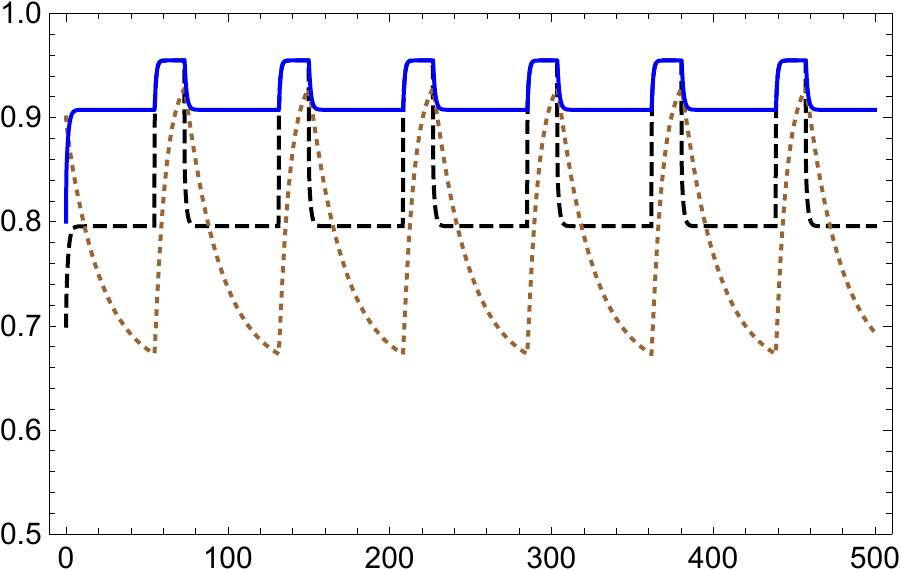} };
  \node[below=of img1, node distance=0cm, yshift=1cm] {$t$};
  \node[left=of img1, node distance=0cm, rotate=90, anchor=center,yshift=-0.7cm] {sine of latitude};
  \node[right=of img1, yshift=0.1cm] (img2)  {\includegraphics[width=1.8in]{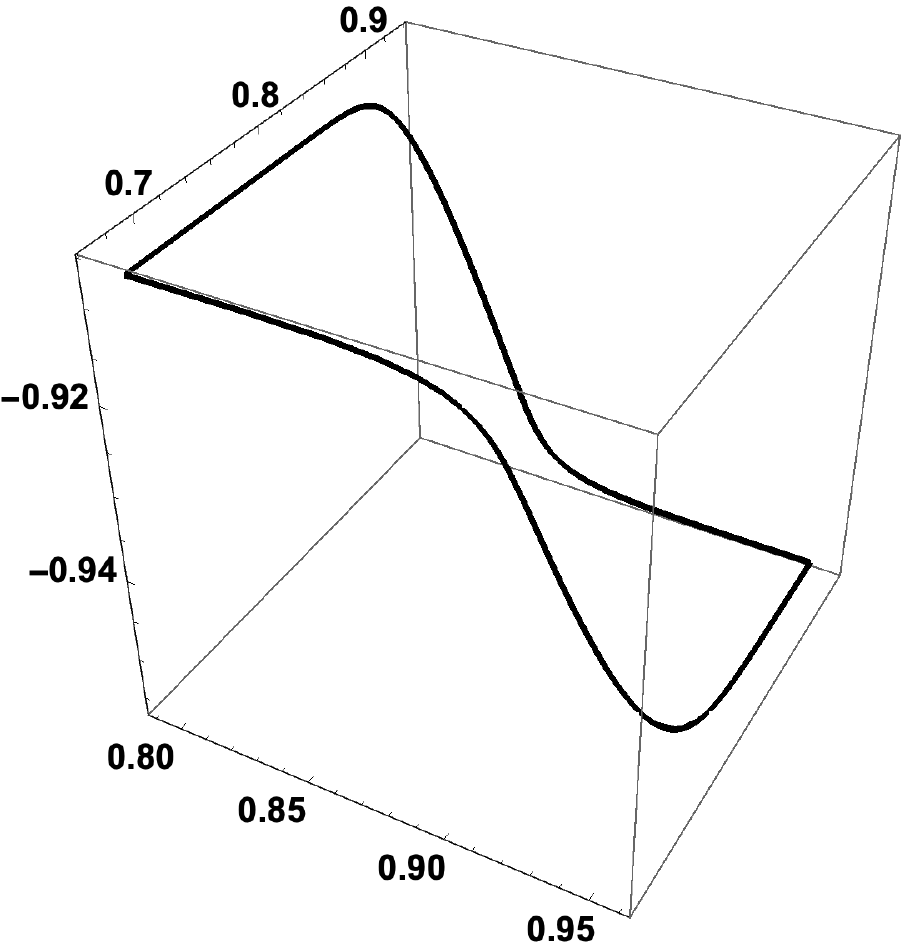}};
  \node[below=of img1, yshift=0.1cm] (img3)  {\includegraphics[width=2.8in]{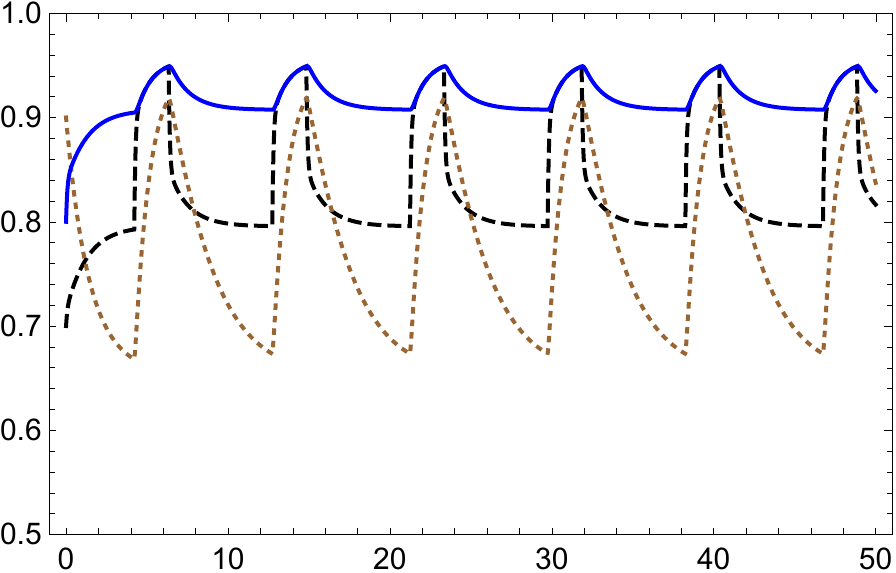}};
  \node[left=of img3, node distance=0cm, rotate=90, anchor=center,yshift=-0.7cm] {sine of latitude};
  \node[below=of img3, node distance=0cm, yshift=1cm] {$t$};
  \node[right=of img3, yshift=0.1cm] (img4)  {\includegraphics[width=1.8in]{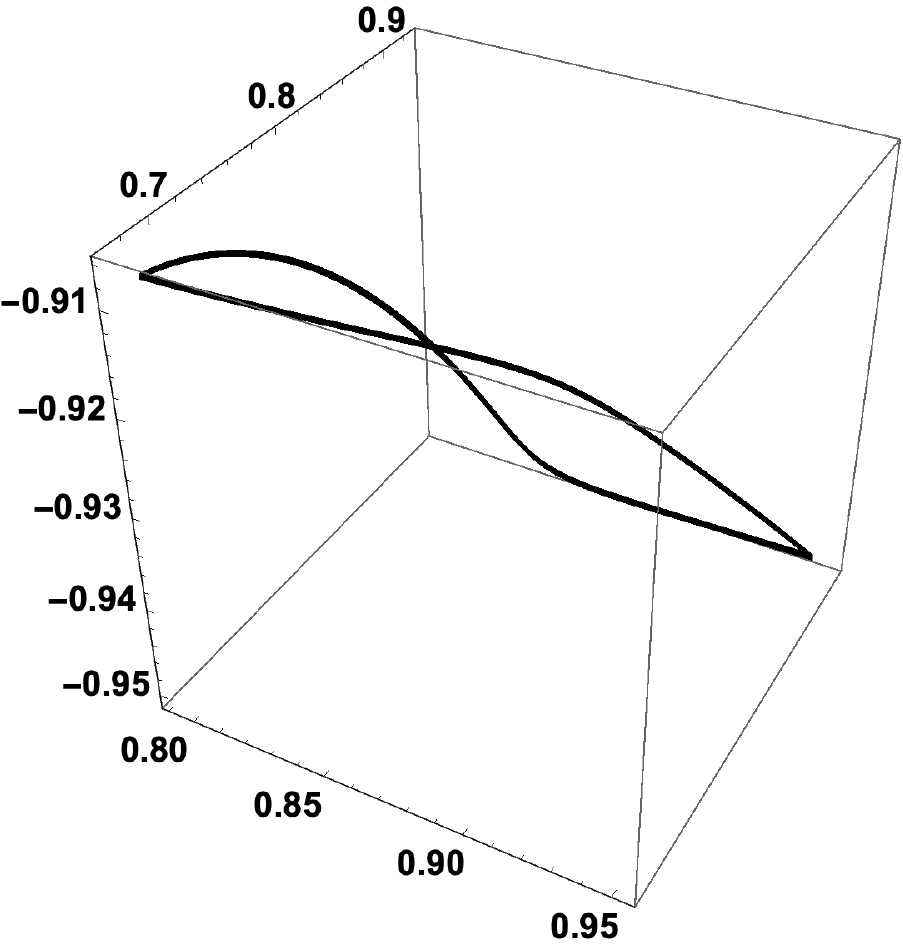}};
  \node[below=of img2, yshift=1.6cm, xshift=-.9cm] {$\eta_N$};
  \node[below=of img2, yshift=3.7cm, xshift=-2.4cm] {$\eta_S$};
  \node[below=of img2, yshift=5.9cm, xshift=-1.5cm] {$\xi_N$};
  \node[below=of img4, yshift=1.6cm, xshift=-.9cm] {$\eta_N$};
  \node[below=of img4, yshift=3.7cm, xshift=-2.4cm] {$\eta_S$};
  \node[below=of img4, yshift=5.9cm, xshift=-1.5cm] {$\xi_N$};
\end{tikzpicture}
\caption{{\small \textbf{Top:} {\em Left}:  The behavior of $\eta_N$ (dashed black curve), $\xi_N$ (dotted brown curve) and $-\eta_S$ (solid blue curve)  along the limit cycle $\Gamma$ when $\Tcnm=-5^\circ$C, $\Tcnp=\Tcs=-10^\circ$C, $\eps=0.03$ and $\rho=0.3$. \ { \em Right}: The projection of $\Gamma$ into $(\eta_S,\eta_N,\xi_N)$-space. \textbf{Bottom:} Same as top row except $\rho=\epsilon=0.3$.
}} 
\label{FIG-neg5-neg10}
 \end{center}
\end{figure}
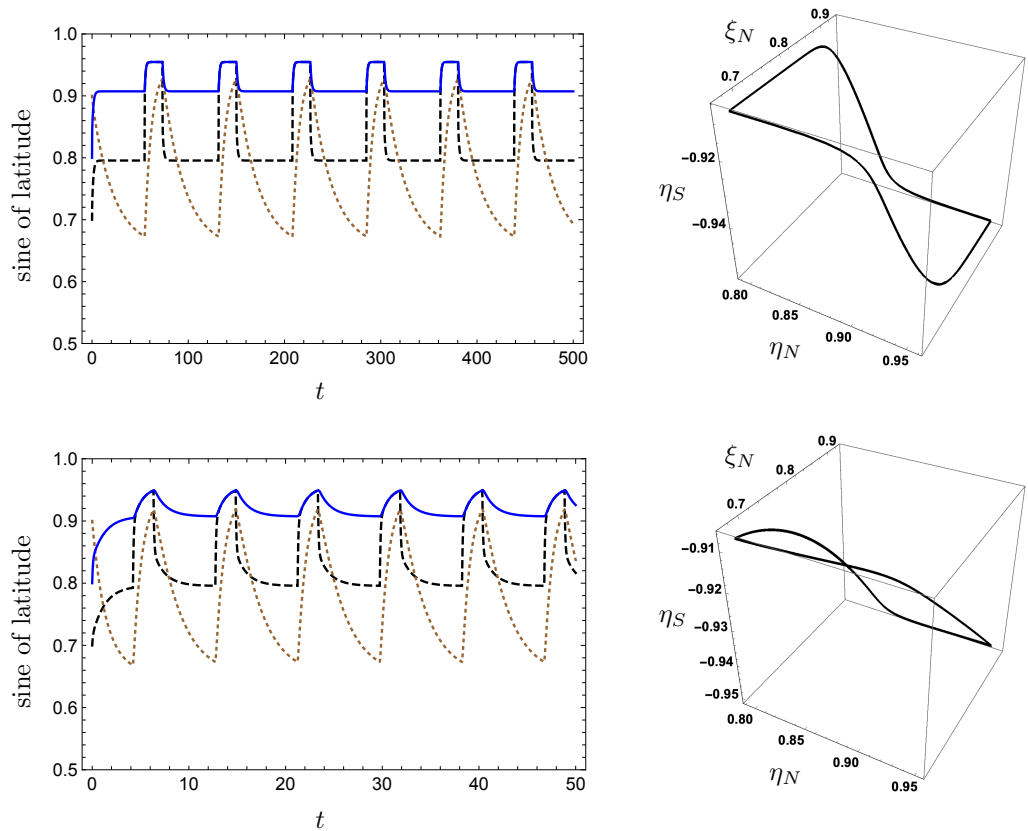

Of particular interest is the oscillation of $\eta_S(t)$ in the Southern Hemisphere, which is completely driven by the ``flip-flop" in the  Northern Hemisphere.  While the $\dot{\eta_S}$-equation has no explicit dependence on $\eta_N$, the dynamic coupling of the hemispheres provided by the $\dot{w}$-equation governs the Southern Hemisphere response to the growth and retreat of the Northern Hemisphere ice sheets. As noted above, this model behavior aligns with theory of M. Milankovitch, which posits that changes in Northern Hemisphere high latitude insolation---due to variations in Earth's orbital elements over long time scales---comprise the principle forcing mechanism of the glacial-interglacial cycles  \cite{hays,milank, raylisnic,uemura}.

We further note the Southern Hemisphere albedo line oscillations (solid blue curve),   while smaller in amplitude, are nonetheless in sync with Northern Hemisphere oscillations (dashed black curve). This behavior is evident in the climate data on orbital time scales \cite{broecker,brook,lowell,raylisnic}. 
The model produces different amplitude oscillations when  choosing different $\Tcnm$-values (see Figure \ref{FIG-Flip-SH}). When the critical temperature during glacial advance is larger, more ice can form and the albedo line advances closer to the equator. Note the effect such a change in the Northern Hemisphere albedo line has on both the mass balance in the Northern Hemisphere and the Southern Hemisphere albedo line. When the critical temperature is more negative, the amplitude of both hemisphere albedo lines and the mass balance is decreased.

For the top row in Figure \ref{FIG-neg5-neg10} the $\eps$-value is an order of magnitude smaller than $\rho$. We numerically find the limit cycle $\Gamma$ exists for larger $\eps$ as well, as illustrated in the bottom row, in which $\eps=\rho$.  

\begin{figure}
\begin{center}
\begin{tikzpicture}
  \node (img1)  {\includegraphics[width=2.8in]{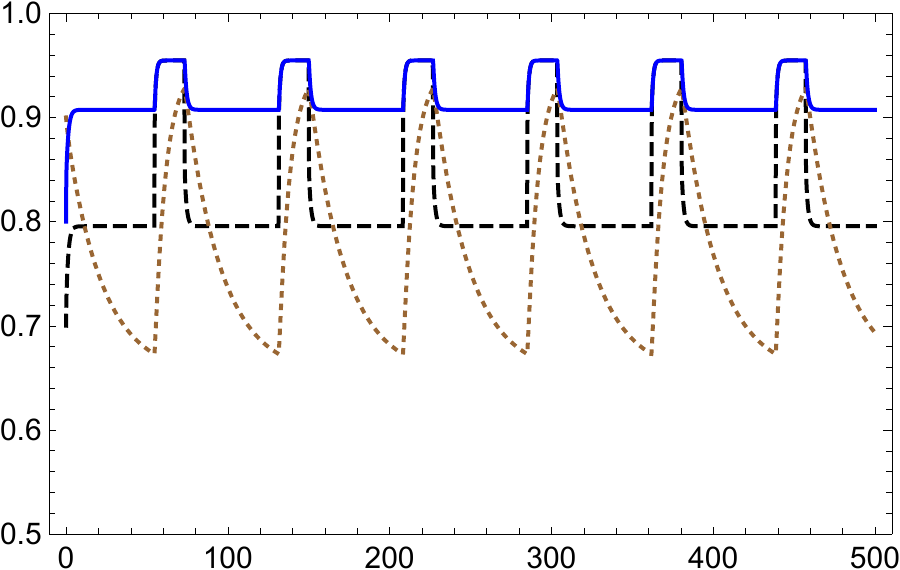}};
  \node[below=of img1, node distance=0cm, yshift=1cm] {$t$};
  \node[left=of img1, node distance=0cm, rotate=90, anchor=center,yshift=-0.7cm] {sine of latitude};
  \node[right=of img1,yshift=0cm] (img2)  {\includegraphics[width=2.8in]{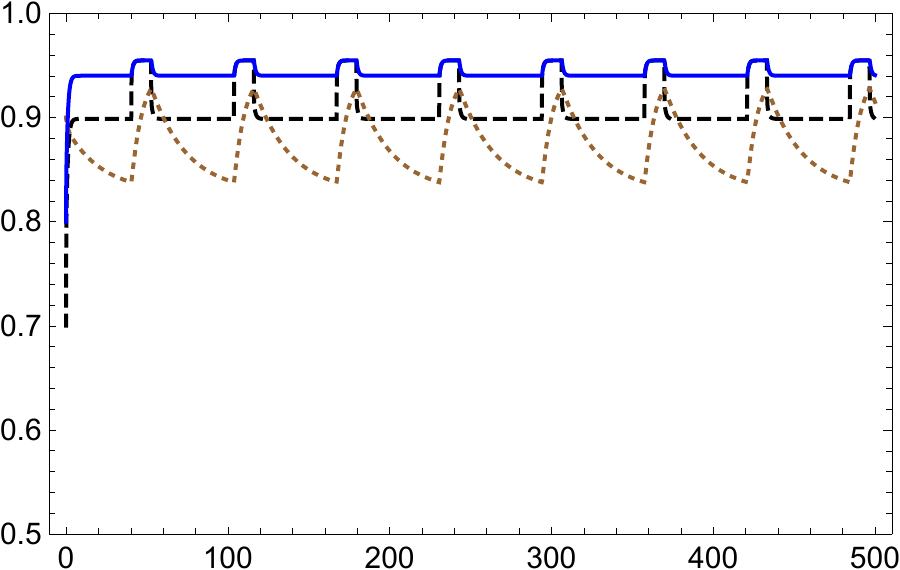}};
  \node[below=of img2, node distance=0cm, yshift=1cm] {$t$};
  \node[left=of img2, node distance=0cm, rotate=90, anchor=center,yshift=-0.7cm] {sine of latitude};
\end{tikzpicture}
\caption{{\small    Periodic behavior of the albedo and ice lines for different critical temperature while the Northern Hemisphere glacier advance.  For all figures $\Tcnp=T_{cS}=-10$ and $\eps=0.03$. Curve coloring and patterns same as in Figure \ref{FIG-neg5-neg10}. \textbf{Left:} $T^-_{cN}=-5^\circ$C and \textbf{Right:} $T^-_{cN}=-8^\circ$C
}} 
\label{FIG-Flip-SH}
 \end{center}
\end{figure}
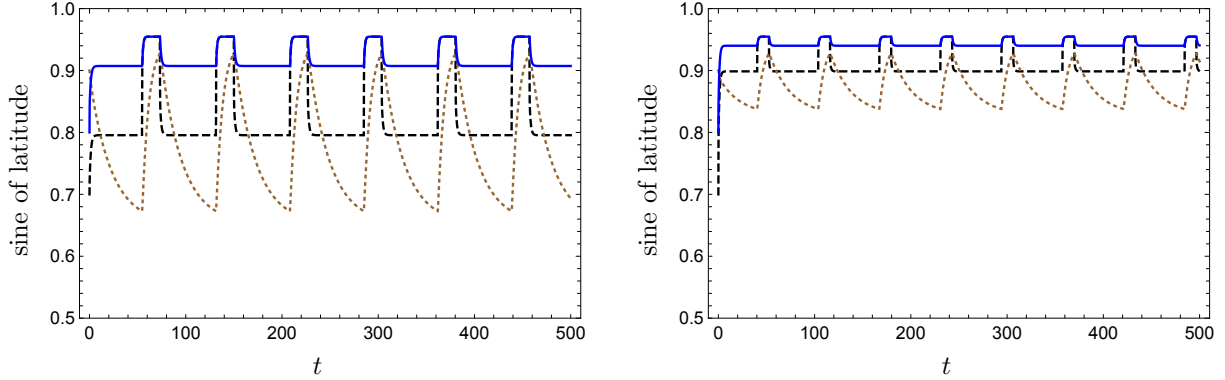

\section{Discussion}
Many planetary energy balance climate models assume a symmetry about the equator; as in Budyko's seminal model for the Earth \cite{budyko}, one focuses solely on the climate in the   Northern Hemisphere.  In this work we couple an approximation of Budyko's latitudinally-averaged surface temperature equation with both Northern Hemisphere and Southern Hemisphere dynamic albedo lines $\eta_N$ and $\eta_S$. Each albedo line is associated with a critical temperature ($T_{cN}, T_{cS}$) that delineates between the local formation and melting of ice.
 
A planet's zonally averaged and mean annual distribution of insolation $s(y,\beta)$ depends on the obliquity $\beta$ as well as the latitude. Earth's obliquity is such that there can only  exist stable, symmetric albedo line positions ($\eta_N=-\eta_S$) if one assumes  $T_{cN}=T_{cS}$. (However, we note that in an energy balance model of Pluto, there exist stable {\em asymmetric} albedo line positions in the case $T_{cN}=T_{cS}$, due to Pluto's obliquity of $119.6^\circ$ \cite{Nadeau2019}.) A full analysis of the temperature-albedo lines $(w,\eta_S,\eta_n)$-system  for Earth in the case $T_{cN}=T_{cS}$, including snowball Earth ($\eta_S=\eta_N$) and ice-free Earth ($\eta_S=-1, \eta_N=1$) scenarios, will appear in a forthcoming paper.

Taking different critical temperature values leads to stable, asymmetric $\eta_S$- and  $\eta_N$-positions in our model. Of particular interest is the fact that a change in $T_{cN}$ alone leads to changes in each of the stable $\eta_S$- and $\eta_N$-albedo line placements, due to the coupling of the Northern Hemisphere and Southern Hemisphere provided by the temperature equation ($\dot{w}$).

The paleoclimate data indicates that on orbital time scales (100 kyr), oscillations in Northern Hemisphere and Southern Hemisphere ice caps are in sync, with evidence suggesting the Southern Hemisphere oscillations are a consequence of changes in the Northern Hemisphere ice sheets  \cite{broecker,brook,lowell,raylisnic}. We were thus lead to incorporate the advance-retreat ``flip-flop" from \cite{wwhm} into the Northern Hemisphere in our model. Using Filippov's theory for discontinuous vector fields, we proved the existence of a unique attracting limit cycle, corresponding to glacial oscillations in which variations in $\eta_S$ and $\eta_N$ are indeed in sync. The cycling of the Northern Hemisphere ice sheet is sufficient to generate changes in the Southern Hemisphere ice extent, again due to the hemispheric coupling inherent in the temperature equation.

The interaction between the hemispheres in our model naturally lends itself to the investigation of several related questions, both of mathematical and paleoclimatic interest. The obliquity, which varies with a period of roughly 41 kyr, can be incorporated into the insolation distribution function $s(y,\beta)$, leading to a nonautonomous and forced discontinuous system. Similarly, the solar ``constant" $Q$ varies with the  eccentricity of Earth's orbit \cite{dickclarence}, and it too can be used to force our Filippov system. As each of the obliquity and eccentricity signals are present in the paleoclimate data \cite{hays}, it would be of interest to analyze the effect of external forcing on the $(w,\eta_S,\eta_N,\xi_N)$-system. Any such study would begin with preliminary investigations into the effect external forcing has on the $(w,\eta_S,\eta_N)$-system \eqref{system-NS}.

Antarctica is believed to have been continuously ice-covered over the past 1 million years   \cite{raylisnic, tzip2003}. In \cite{raylisnic}, Raymo et al posit that Antarctica's ice sheet was more dynamic 3 million years ago (mya), with a terrestrial-based ice margin. It is further suggested in \cite{raylisnic} that the transition from a dynamic to a permanent Antarctic ice sheet played an important role in the Mid-Pleistocene Transition, a time roughly 1 mya in which the period of the glacial cycles changed from 41 kyr to 100 kyr. This conceptual scenario can be investigated  with our model, perhaps with the addition of a $\xi_S$-variable for the period of time when Antarctic ice terminated on land, and which over time coalesced with the albedo line $\eta_S$ as a parameter varies. More generally, the use of nonsmooth bifurcation theory as a tool to investigate changes in our system as various parameters vary is easy to envision.

Finally, it is of interest to note there is an asymmetry in Northern Hemisphere and Southern Hemisphere glacial cycle oscillations on a millenial time scale---the so-called {\em bipolar seesaw} \cite{broecker, brook, pedro}. This asymmetry is thought to be caused by disruptions in the meridional transport of heat by the ocean up to the North Atlantic, due in turn to the discharge of fresh meltwater from Northern Hemisphere ice sheets into the North Atlantic ocean \cite{timmermann}. In terms of the model, this ocean heat transport is associated with the $-C(T-\overline{T})$-term, which concerns the global average surface temperature. In our view a diffusive meridional heat transport approach \cite{north,sellers,Walsh2020} would be more appropriate for investigations into millenial time scale, asynchronous oscillations in Northern Hemisphere and Southern Hemisphere ice extent, perhaps incorporating the diffusion coefficient as a function of latitude and thereby bringing into play localized heat transport.

\section*{Acknowledgements}
Research of AN was supported by an NSF Mathematical Sciences Postdoctoral Research Fellowship, Award Number DMS-190288.

\bibliographystyle{siam}
\bibliography{budyko}

\begin{thebibliography}{10}

\bibitem{abe-ouchi}
{\sc A.~Abe-Ouchi, F.~Saito, K.~Kawamura, M.~E. Raymo, J.~Okuno, K.~Takahashi,
  and H.~Blatter}, {\em Insolation-driven 100,000-year glacial cycles and
  hysteresis of ice-sheet volume}, Nature, 500 (2013), pp.~190--194.

\bibitem{awrejcewicz}
{\sc J.~Awrejcewicz, M.~Fe{\v{c}}kan, and P.~Olejnik}, {\em On continuous
  approximation of discontinuous systems}, Nonlinear Analysis: Theory, Methods
  \& Applications, 62 (2005), pp.~1317--1331.

\bibitem{Barry2017}
{\sc A.~Barry, E.~Widiasih, and R.~McGehee}, {\em {Nonsmooth frameworks for an
  extended Budyko model}}, Discrete and Continuous Dynamical Systems B, 22
  (2017), pp.~2447--2463.

\bibitem{blunier}
{\sc T.~Blunier and E.~J. Brook}, {\em {Timing of millennial-scale climate
  change in Antarctica and Greenland during the last glacial period}}, Science,
  291 (2001), pp.~109--112.

\bibitem{broecker}
{\sc W.~S. Broecker}, {\em {Thermohaline circulation, the Achilles heel of our
  climate system: Will man-made CO2 upset the current balance?}}, Science, 278
  (1997), pp.~1582--1588.

\bibitem{brook}
{\sc E.~J. Brook and C.~Buizert}, {\em Antarctic and global climate history
  viewed from ice cores}, Nature, 558 (2018), pp.~200--208.

\bibitem{budyko}
{\sc M.~I. Budyko}, {\em {The effect of solar radiation variations on the
  climate of the Earth}}, Tellus, 21 (1969), pp.~611--619.

\bibitem{choudhury}
{\sc D.~Choudhury, A.~Timmermann, F.~Schloesser, M.~Heinemann, and D.~Pollard},
  {\em {Simulating Marine Isotope Stage 7 with a coupled climate--ice sheet
  model}}, Climate of the Past, 16 (2020), pp.~2183--2201.

\bibitem{darvill}
{\sc C.~M. Darvill, M.~J. Bentley, C.~R. Stokes, and J.~Shulmeister}, {\em {The
  timing and cause of glacial advances in the southern mid-latitudes during the
  last glacial cycle based on a synthesis of exposure ages from Patagonia and
  New Zealand}}, Quaternary Science Reviews, 149 (2016), pp.~200--214.

\bibitem{dibernardo}
{\sc M.~di~Bernardo, C.~Budd, A.~R. Champneys, and P.~Kowalczyk}, {\em
  Piecewise-smooth dynamical systems: theory and applications}, vol.~163,
  Springer-Verlag, London, UK, 2008.

\bibitem{dieci}
{\sc L.~Dieci, C.~Elia, and D.~Pi}, {\em Limit cycles for regularized
  discontinuous dynamical systems with a hyperplane of discontinuity}, Discrete
  \& Continuous Dynamical Systems-B, 22 (2017), p.~3091.

\bibitem{engler2018dynamical}
{\sc H.~Engler, H.~Kaper, T.~Kaper, and T.~Vo}, {\em {A dynamical systems
  approach to the Pleistocene climate}}, EGUGA,  (2018), p.~5085.

\bibitem{fil}
{\sc A.~F. Filippov}, {\em Differential equations with discontinuous right-hand
  side}, Amer. Math. Soc. Trans., 42 (1964), pp.~199--231.

\bibitem{fraser}
{\sc C.~I. Fraser, R.~Nikula, H.~G. Spencer, and J.~M. Waters}, {\em {Kelp
  genes reveal effects of subantarctic sea ice during the Last Glacial
  Maximum}}, Proceedings of the National Academy of Sciences, 106 (2009),
  pp.~3249--3253.

\bibitem{freire}
{\sc E.~Freire, E.~Ponce, F.~Rodrigo, and F.~Torres}, {\em Bifurcation sets of
  continuous piecewise linear systems with two zones}, International Journal of
  Bifurcation and Chaos, 8 (1998), pp.~2073--2097.

\bibitem{gallee}
{\sc H.~Gall{\'e}e, J.~Van~Yperselb, T.~Fichefet, I.~Marsiat, C.~Tricot, and
  A.~Berger}, {\em {Simulation of the last glacial cycle by a coupled,
  sectorially averaged climate-ice sheet model: 2. Response to insolation and
  CO$_2$ variations}}, Journal of Geophysical Research: Atmospheres, 97 (1992),
  pp.~15713--15740.

\bibitem{gersonde}
{\sc R.~Gersonde, X.~Crosta, A.~Abelmann, and L.~Armand}, {\em {Sea-surface
  temperature and sea ice distribution of the Southern Ocean at the EPILOG Last
  Glacial Maximum—a circum-Antarctic view based on siliceous microfossil
  records}}, Quaternary Science Reviews, 24 (2005), pp.~869--896.

\bibitem{hays}
{\sc J.~D. Hays, J.~Imbrie, N.~J. Shackleton, et~al.}, {\em {Variations in the
  Earth’s orbit: pacemaker of the ice ages}}, Science, 194 (1976),
  pp.~1121--1132.

\bibitem{huybers2011}
{\sc P.~Huybers}, {\em {Combined obliquity and precession pacing of late
  Pleistocene deglaciations}}, Nature, 480 (2011), pp.~229--232.

\bibitem{huybers2005}
{\sc P.~Huybers and C.~Wunsch}, {\em {Obliquity pacing of the late Pleistocene
  glacial terminations}}, Nature, 434 (2005), pp.~491--494.

\bibitem{jones}
{\sc C.~K. Jones}, {\em Geometric singular perturbation theory}, in Dynamical
  systems, L.~Arnold, ed., Springer-Verlag, Berlin, Germany, 1995, pp.~44--118.

\bibitem{Kaper2013}
{\sc H.~Kaper and H.~Engler}, {\em {Mathematics and Climate}}, SIAM,
  Philadelphia, PA, 2013.

\bibitem{kawamura}
{\sc K.~Kawamura, A.~Abe-Ouchi, H.~Motoyama, Y.~Ageta, S.~Aoki, N.~Azuma,
  Y.~Fujii, K.~Fujita, S.~Fujita, K.~Fukui, et~al.}, {\em {State dependence of
  climatic instability over the past 720,000 years from Antarctic ice cores and
  climate modeling}}, Science advances, 3 (2017), p.~e1600446.

\bibitem{knutti}
{\sc R.~Knutti, J.~Fl{\"u}ckiger, T.~Stocker, and A.~Timmermann}, {\em Strong
  hemispheric coupling of glacial climate through freshwater discharge and
  ocean circulation}, Nature, 430 (2004), pp.~851--856.

\bibitem{julie}
{\sc J.~Leifeld}, {\em Perturbation of a nonsmooth supercritical hopf
  bifurcation}, tech. rep., arXiv: 1601.07930, 2016.

\bibitem{julie-dis}
\leavevmode\vrule height 2pt depth -1.6pt width 23pt, {\em {Smooth and
  Nonsmooth Bifurcations in Welander's Ocean Convection Model}}, PhD thesis,
  University of Minnesota, University of Minnesota Digital Conservancy, 2016.
\newblock http://hdl.handle.net/11299/182310.

\bibitem{leine}
{\sc R.~I. Leine and H.~Nijmeijer}, {\em Dynamics and bifurcations of
  non-smooth mechanical systems}, Springer-Verlag, Berlin, Germany, 2004.

\bibitem{llibre2017}
{\sc J.~Llibre and M.~A. Teixeira}, {\em Piecewise linear differential systems
  without equilibria produce limit cycles?}, Nonlinear dynamics, 88 (2017),
  pp.~157--164.

\bibitem{lowell}
{\sc T.~Lowell, C.~Heusser, B.~Andersen, P.~Moreno, A.~Hauser, L.~Heusser,
  C.~Schl{\"u}chter, D.~Marchant, and G.~Denton}, {\em {Interhemispheric
  correlation of late Pleistocene glacial events}}, Science, 269 (1995),
  pp.~1541--1549.

\bibitem{dickclarence}
{\sc R.~McGehee and C.~Lehman}, {\em {A paleoclimate model of ice-albedo
  feedback forced by variations in Earth's orbit}}, SIAM Journal on Applied
  Dynamical Systems, 11 (2012), pp.~684--707.

\bibitem{McGehee2014}
{\sc R.~McGehee and E.~Widiasih}, {\em {A quadratic approximation to Budyko's
  ice-albedo feedback model with ice line dynamics}}, SIAM Journal on Applied
  Dynamical Systems, 13 (2014), pp.~518--536.

\bibitem{milank}
{\sc M.~Milankovitch}, {\em {Canon of insolation and the ice-age problem (Kanon
  der Erdbestrahlung und seine Anwendung auf das Eiszeitenproblem) Belgrade,
  1941.}}, Israel Program for Scientific Translations,  (1969).

\bibitem{budd}
{\sc K.~S. Morupisi and C.~J. Budd}, {\em {An analysis of the periodically
  forced PP04 climate model, using the theory of non-smooth dynamical
  systems}}, IMA Journal of Applied Mathematics (to appear),  (2020).

\bibitem{Nadeau-dis}
{\sc A.~Nadeau}, {\em {Generalizations for Insolation and Albedo to Adapt an
  Energy Balance Model to Other Planets}}, PhD thesis, University of Minnesota,
  University of Minnesota Digital Conservancy, 2019.
\newblock http://hdl.handle.net/11299/206425.

\bibitem{Nadeau2019}
{\sc A.~Nadeau and E.~Jaschke}, {\em {Stable asymmetric ice belts in an energy
  balance model of Pluto}}, Icarus, 331 (2019), pp.~15--25.

\bibitem{Nadeau2017}
{\sc A.~Nadeau and R.~McGehee}, {\em A simple formula for a planet’s mean
  annual insolation by latitude}, Icarus, 291 (2017), pp.~46--50.

\bibitem{north}
{\sc G.~R. North}, {\em Theory of energy-balance climate models}, Journal of
  the Atmospheric Sciences, 32 (1975), pp.~2033--2043.

\bibitem{pedro}
{\sc J.~B. Pedro, M.~Jochum, C.~Buizert, F.~He, S.~Barker, and S.~O.
  Rasmussen}, {\em {Beyond the bipolar seesaw: Toward a process understanding
  of interhemispheric coupling}}, Quaternary Science Reviews, 192 (2018),
  pp.~27--46.

\bibitem{peltier}
{\sc W.~R. Peltier and S.~Marshall}, {\em {Coupled energy-balance/ice-sheet
  model simulations of the glacial cycle: A possible connection between
  terminations and terrigenous dust}}, Journal of Geophysical Research:
  Atmospheres, 100 (1995), pp.~14269--14289.

\bibitem{pierre}
{\sc R.~T. Pierrehumbert}, {\em Climate dynamics of a hard snowball earth},
  Journal of Geophysical Research: Atmospheres, 110 (2005).

\bibitem{raylisnic}
{\sc M.~E. Raymo, L.~Lisiecki, and K.~H. Nisancioglu}, {\em {Plio-Pleistocene
  ice volume, Antarctic climate, and the global $\delta^{18}$O record}},
  Science, 313 (2006), pp.~492--495.

\bibitem{rother}
{\sc H.~Rother, D.~Fink, J.~Shulmeister, C.~Mifsud, M.~Evans, and J.~Pugh},
  {\em {The early rise and late demise of New Zealand’s last glacial
  maximum}}, Proceedings of the National Academy of Sciences, 111 (2014),
  pp.~11630--11635.

\bibitem{saltzman}
{\sc B.~Saltzman, A.~R. Hansen, and K.~A. Maasch}, {\em {The late Quaternary
  glaciations as the response of a three-component feedback system to
  Earth-orbital forcing}}, Journal of the Atmospheric Sciences, 41 (1984),
  pp.~3380--3389.

\bibitem{sellers}
{\sc W.~D. Sellers}, {\em A global climatic model based on the energy balance
  of the earth-atmosphere system}, Journal of Applied Meteorology, 8 (1969),
  pp.~392--400.

\bibitem{sotomayor1996}
{\sc J.~Sotomayor and M.~Teixeira}, {\em Regularization of discontinuous vector
  fields}, in Proceedings of the international conference on differential
  equations, Lisboa, World Scientific, 1996, pp.~207--223.

\bibitem{timmermann}
{\sc A.~Timmermann, H.~Gildor, M.~Schulz, and E.~Tziperman}, {\em Coherent
  resonant millennial-scale climate oscillations triggered by massive meltwater
  pulses}, Journal of Climate, 16 (2003), pp.~2569--2585.

\bibitem{Tung2007}
{\sc K.-K. Tung}, {\em {Topics in Mathematical Modeling}}, Princeton University
  Press, Princeton, NJ, 2007.

\bibitem{tzip2003}
{\sc E.~Tziperman and H.~Gildor}, {\em {On the mid-Pleistocene transition to
  100-kyr glacial cycles and the asymmetry between glaciation and deglaciation
  times}}, Paleoceanography, 18 (2003).

\bibitem{uemura}
{\sc R.~Uemura, H.~Motoyama, V.~Masson-Delmotte, J.~Jouzel, K.~Kawamura,
  K.~Goto-Azuma, S.~Fujita, T.~Kuramoto, M.~Hirabayashi, T.~Miyake, et~al.},
  {\em {Asynchrony between Antarctic temperature and CO$_2$ associated with
  obliquity over the past 720,000 years}}, Nature Communications, 9 (2018),
  pp.~1--11.

\bibitem{Walsh2015}
{\sc J.~Walsh and C.~Rackauckas}, {\em {On the Budyko-Sellers energy balance
  climate model with ice line coupling}}, Discrete \& Continuous Dynamical
  Systems-B, 20 (2015), pp.~2187--2216.

\bibitem{Walsh2020}
{\sc J.~Walsh and E.~Widiasih}, {\em {A discontinuous ODE model of the glacial
  cycles with diffusive heat transport}}, Mathematics, 8 (2020).

\bibitem{wwhm}
{\sc J.~Walsh, E.~Widiasih, J.~Hahn, and R.~McGehee}, {\em Periodic orbits for
  a discontinuous vector field arising from a conceptual model of glacial
  cycles}, Nonlinearity, 29 (2016), pp.~1843--1864.

\bibitem{weertman}
{\sc J.~Weertman}, {\em Milankovitch solar radiation variations and ice age ice
  sheet sizes}, Nature, 261 (1976), pp.~17--20.

\bibitem{welander}
{\sc P.~Welander}, {\em A simple heat-salt oscillator}, Dynamics of Atmospheres
  and Oceans, 6 (1982), pp.~233--242.

\bibitem{Widiasih2013}
{\sc E.~Widiasih}, {\em {Dynamics of the Budyko energy balance model}}, SIAM
  Journal of Applied Dynamical Systems, 12 (2013), pp.~2068--2092.

\bibitem{wright}
{\sc H.~Wright and I.~Stefanova}, {\em Plant trash in the basal sediments of
  glacial lakes}, Acta Palaeobotanica, 44 (2004), pp.~141--146.

\end{thebibliography}

\end{document}